\documentclass[11pt]{amsart}
\usepackage{amssymb,amsmath,amsthm,amsfonts,mathrsfs}
\usepackage[hmargin=3cm,vmargin=3.5cm]{geometry}
\usepackage[dvipsnames,table,xcdraw]{xcolor}
\usepackage[colorlinks = true,
            linkcolor = ForestGreen,
            urlcolor  = Fuchsia,
            citecolor = blue,
            anchorcolor = blue]{hyperref}

\usepackage{stackengine}

\usepackage[all,2cell]{xy} \UseAllTwocells 
\usepackage{tikz-cd}
\usepackage{quiver} 

\usepackage{color}
\usepackage{graphicx,psfrag}
\usepackage{amscd}  
\usepackage[all,2cell]{xy} \UseAllTwocells \SilentMatrices
\usepackage{tikz}
\usepackage[dvips]{epsfig}
\usepackage{comment}
\usepackage[shortlabels]{enumitem}
\usepackage{kbordermatrix}
\usepackage{cancel,scalefnt}
\usepackage{bm}
 
\usepackage{array}

\newcolumntype{L}{>{$}l<{$}} 

\theoremstyle{definition}
\newtheorem{theorem}{Theorem}[section]
\newtheorem{lemma}[theorem]{Lemma}
\newtheorem{remark}[theorem]{Remark}

\newtheorem{definition}[theorem]{Definition}

\newtheorem{proposition}[theorem]{Proposition}

\newtheorem{corollary}[theorem]{Corollary}
\newtheorem{example}[theorem]{Example}

\catcode`~=11 
\newcommand{\urltilde}{\kern -.15em\lower .7ex\hbox{~}\kern .04em}  
\catcode`~=13 

\usetikzlibrary{arrows,automata}
\usetikzlibrary{decorations.markings}
\usetikzlibrary{decorations.pathmorphing,shapes,decorations.text,shapes.geometric}
\usetikzlibrary{shadings}
\usepackage{multirow}

\newcounter{sarrow}

\usepackage[colorinlistoftodos]{todonotes}

\definecolor{darkred}{rgb}{0.7,0,0} 
\newcommand{\defn}[1]{{\color{darkred}\emph{#1}}} 

\title[Multisymmetric polynomials on set-theoretic quiver representations]{Multisymmetric polynomials on set-theoretic quiver representations}

\author[Radford Green]{Radford Green}
\address{Department of Mathematics, Johns Hopkins University, Baltimore, MD 21218, USA}
\email{\href{mailto:rgreen87@jh.edu}{rgreen87@jh.edu}}

\author[Cornell Holmes]{Cornell Holmes}
\address{Department of Mathematics, Johns Hopkins University, Baltimore, MD 21218, USA}
\email{\href{mailto:cholme21@jhu.edu}{cholme21@jhu.edu}}

\author[Mee Seong Im]{Mee Seong Im}
\address{Department of Mathematics, Johns Hopkins University, Baltimore, MD 21218, USA}
\email{\href{mailto:meeseong@jhu.edu}{meeseong@jhu.edu}}

\makeatletter
\@namedef{subjclassname@2020}{%
  \textup{2020} Mathematics Subject Classification}


\providecommand{\keywords}[1]{\textbf{\textit{Key words and phrases.}} #1}


\date{\today}
\begin{document}

\def\mfb{\mathfrak{b}}

\def\A{\mathbb{A}}

\def\E{\mathsf E}
\def\F{\mathbb{F}}
\def\In{\mathsf{in}}
\def\I{\mathsf I}
\def\R{\mathbb R}
\def\Q{\mathbb Q}
\def\Z{\mathbb Z}
\def\map{\mathsf{map}}
\def\N{\mathbb N}
\def\P{\mathsf{P}}
\def\C{\mathbb C}
\def\S{\mathbb S}
\def\Kr{\mathsf{Kr}}
\def\sgn{\mathsf{sgn}}
\def\t{\mathsf{t}}
\def\target{\mathsf{target}}
\def\DAG{\mathsf{DAG}}
\def\Set{\mathsf{Set}}
\def\Frac{\mathsf{Frac}}
\def\FinSet{\mathsf{FinSet}}
\def\Sym{\mathsf{Sym}}
\def\Leinster{\mathsf{Leinster}}
\def\Lin{\mathsf{Lin}}
\def\Nil{\mathsf{Nil}}
\def\SS{\mathbb S}
\def\GL{\mathsf{GL}}
\def\Graph{\mathsf{Graph}}
\def\top{\mathsf{top}}

\def\for{\mathsf{for}}
\def\endo{\mathsf{end}}
\def\hom{\mathsf{hom}}
\def\Hom{\mathsf{Hom}}
\def\End{\mathsf{End}}
\newcommand{\card}{\mathsf{card}}
\newcommand{\Rep}{\mathsf{SetRep}}

\def\Fib{\mathsf{Fib}}
\def\Nss{\mathsf{Nss}}
\def\Src{\mathsf{Src}}
\def\Snk{\mathsf{Snk}}
\def\Succ{\mathsf{Succ}}

\def\Der{\mathsf{Der}}
\def\Pol{\mathsf{Pol}}

\newcommand{\dmod}{\mathsf{-mod}}
\newcommand{\comp}{\mathrm{comp}} 
\newcommand{\col}{\mathrm{col}}
\newcommand{\adm}{\mathrm{adm}}  
\newcommand{\Ob}{\mathrm{Ob}}
\newcommand{\Cob}{\mathsf{Cob}}
\newcommand{\UCob}{\mathsf{UCob}}
\newcommand{\COB}{\mathcal{COB}}
\newcommand{\ECob}{\mathsf{ECob}}
\newcommand{\id}{\mathsf{id}}
\newcommand{\undM}{\underline{M}}
\newcommand{\im}{\mathsf{im}\:}
\newcommand{\coker}{\mathsf{coker}}
\newcommand{\Aut}{\mathsf{Aut}}
\newcommand{\tripod}{\mathsf{Td}}
\newcommand{\BBC}{\mathbb{B}(\mathcal{C})}
\newcommand{\Pmod}{\mathrm{pmod}}
\newcommand{\gammaoneR}{\gamma_{1,R}}  
\newcommand{\gammaoneRbar}{\overline{\gamma}_{1,R}} 
\newcommand{\gammaoneRprime}
{\gamma'_{1,R}}
\newcommand{\gammaoneRbarprime}
{\overline{\gamma}'_{1,R}}
\newcommand{\qbinom}[3]{\genfrac{[}{]}{0pt}{}{#1}{#2}_{#3}}
\newcommand{\thick}[1]{\pmb{\mathbf{#1}}}

\def\l{\lbrace}
\def\r{\rbrace}
\def\o{\otimes}
\def\out{\mathsf{out}}
\def\lra{\longrightarrow}
\def\ed{\mathsf{ed}}
\def\Ext{\mathsf{Ext}}
\def\ker{\mathsf{ker}}
\def\mf{\mathfrak} 
\def\mcC{\mathcal{C}}
\def\mcS{\mathcal{S}}  
\def\mcQC{\mathcal{QC}}
\def\mcA{\mathcal{A}}
\def\mcF{\mathcal{F}}
\def\mcE{\mathcal{E}}
\def\Fr{\mathsf{Fr}}  

\def\bbn{\mathbb{B}^n}
\def\ovb{\overline{b}}
\def\tr{{\sf tr}} 
\def\det{{\sf det }} 
\def\one{\mathbf{1}}   
\def\kk{\mathbf{k}}  
\def\gdim{\mathsf{gdim}}  
\def\rk{\mathsf{rk}}
\def\EC{\mathsf{EC}}
\def\WEC{\mathsf{WEC}}
\def\CEC{\mathsf{CEC}}
\def\IET{\mathsf{IET}}
\def\SAF{\mathsf{SAF}}

\newcommand{\indexw}{\R_{>0}} 

\newcommand{\brak}[1]{\ensuremath{\left\langle #1\right\rangle}}
\newcommand{\oplusop}[1]{{\mathop{\oplus}\limits_{#1}}}
\newcommand{\addfigure}{\vspace{0.1in} \begin{center} {\color{red} ADD FIGURE} \end{center} \vspace{0.1in} }
\newcommand{\add}[1]{\vspace{0.1in} \begin{center} {\color{red} ADD FIGURE #1} \end{center} \vspace{0.1in} }
\newcommand{\vspin}{\vspace{0.1in} }

\newcommand\circled[1]{\tikz[baseline=(char.base)]{\node[shape=circle,draw,inner sep=1pt] (char) {${#1}$};}} 

\let\oldemptyset\emptyset
\let\emptyset\varnothing

\let\oldtocsection=\tocsection
\let\oldtocsubsection=\tocsubsection
\renewcommand{\tocsection}[2]{\hspace{0em}\oldtocsection{#1}{#2}}
\renewcommand{\tocsubsection}[2]{\hspace{1em}\oldtocsubsection{#1}{#2}}

\renewcommand{\kbldelim}{(}
\renewcommand{\kbrdelim}{)}

\def\MK#1{{\color{red}[MK: #1]}}
\def\bfred#1{{\color{red}#1}}


\begin{abstract}
Eventually constant set-valued representations of a quiver are set-theoretic analogues of nilpotent representations. In recent work by Green--Holmes--Im, the authors enumerated eventually constant set-valued representations for equioriented cyclic quivers using the directed matrix-tree theorem. In this paper, we extend this enumeration to finite quivers without sinks for which every vertex is the target of sufficiently long paths. We encode the representations as directed acyclic graphs and introduce a recursive source-removal method for certain classes of directed acyclic graphs. This yields a strictly upper triangular matrix enumerator in the incidence algebra of the subset lattice. To compute the cardinality of the eventually constant representations, we compress this enumerator to a matrix indexed by cardinality vectors, the set-theoretic analogues of dimension vectors. We conclude by specializing the formulas to the Jordan quiver and recovering the multisymmetric generating polynomial for the cyclic quiver without using the matrix-tree theorem.
\end{abstract}

\maketitle
\tableofcontents

%
%

\section{Introduction}
\label{section_intro} 
Cayley's formula states that the number of unrooted trees with vertex set $V$ of cardinality $n$ is $n^{n-2}$. Rooted trees are trees together with a choice of a vertex (root), and there are $n^{n-1}$ of them. 
Tom Leinster showed eventually constant functions correspond to rooted trees in~\cite{Lei21}, providing a conceptual proof that avoids calculation and algebraic manipulation of $q$-binomial coefficients and algebraic partitions. An endomorphism $f:V\to V$ of a finite set is \defn{eventually constant} if its iterated compositions $f^m=f\circ\cdots\circ f$ eventually stabilize to a constant map. Since there are $n^{n}$ self-maps of $V$, this shows that the probability that a function is eventually constant is $1/n$.
In~\cite{GHI26,CIKLR25,CILR25}, the authors generalize the notion of eventual constancy. Instead of a single endomorphism $f$ on one set $V$, they consider cyclically oriented $k$-tuples $(f_1,\ldots,f_k)$ of maps across $k$ disjoint nonempty finite sets $V_1,\ldots,V_k$, where 
\[ 
V_1\xrightarrow{f_1}V_2\xrightarrow {f_2}\ldots\xrightarrow{f_{k-1}}V_k\xrightarrow{f_k}V_1. 
\]
In this setting, subscripts are taken modulo $k$, and a \defn{transversal} $C\subseteq\coprod_{i=1}^kV_i$ is a subset that contains a single element $c_i$ in each set $V_i$. A $k$-tuple $(f_1,\ldots,f_k)$ of maps is \defn{eventually constant} if there exists a transversal $C$ such that all sufficiently long cyclic compositions $f_{i+n}\circ f_{i+n-1}\circ\cdots\circ f_{i+1}\circ f_i$ map into $C$.

To derive a weighted enumeration of the eventually constant $k$-tuples, we identify each $k$-tuple with an oriented graph. Each oriented graph has a unique cycle along a transversal $C\subseteq\coprod_{i=1}^kV_i$, and removing its edges produces a forest rooted at $C$. The directed matrix-tree theorem~\cite{CK78} gives a determinantal formula for the weighted count of the resulting set of forests rooted at each transversal, and summing the cycle-weighted counts of the forests over all transversals yields the enumeration.

In this paper, we generalize eventual constancy from cyclically oriented maps to collections of sets and functions between them in any orientation, and then construct matrix enumerators. These collections of sets and oriented functions are the set-theoretic analogues of quiver representations,~\cite{DW05,DW17}. Furthermore, eventually constant set-theoretic representations are analogous to nilpotent linear representations of quivers~\cite{Ls91}.

Fix a \defn{quiver} $Q:=(Q_0,Q_1)$, a directed graph, where $Q_0$ is a finite set of vertices and $Q_1$ is a finite set of arrows between them. We assume that $Q$ has no sinks and that all vertices are the target of arbitrarily long paths.
Classically, a linear \defn{representation} $W$ of $Q$ is a collection $\{ W_a \}_{a\in Q_0}$ of finite dimensional vector spaces indexed by $a\in Q_0$, and a set $\{W_\alpha:W_{a}\to W_{b}\mid \alpha:a\to b\}$ of linear maps indexed by $\alpha\in Q_1$.
Passing to the set-theoretic analogue, a \defn{set-valued representation} replaces each finite dimensional vector space $W_a$ with a finite set, and replaces each linear map $W_\alpha:W_a\to W_b$ with a function. Equivalently, a linear representation is a functor from the path category on $Q$ to the category of finite dimensional vector spaces, and a set-valued representation is a functor from this path category to the category of finite sets. 

In place of the sets $V_1,\ldots,V_k$ from the cyclic case, we fix a tuple $V:=(V_a)_{a\in Q_0}$ of nonempty finite disjoint sets. Let $\overline V:=\coprod_{a\in Q_0}V_a$. Instead of the equioriented maps $(f_1,\ldots,f_k)$, we study the collection of set-valued representations of $Q$ with the underlying $Q_0$-indexed set $V$. A set-valued representation $W$ belongs to $\Rep(Q,V)$ if $W_a=V_a$ for each $a\in Q_0$. A classical representation is \defn{nilpotent} if all sufficiently long compositions are trivial. Since sets lack the zero object, we exchange nilpotency for constancy, and define a \defn{transversal} of $V$ as a subset $C\subseteq\overline V$ that contains a single element $c_a$ in each set $V_a$. A representation $W\in\Rep(Q,V)$ is \defn{eventually constant} if there exists a transversal $C$ such that for all sufficiently long paths $\alpha_\ell\cdots\alpha_1$ in $Q$, the composition $W_{\alpha_\ell}\circ\cdots\circ W_{\alpha_1}$ maps into $C$. Let $\EC$ denote the subset of eventually constant representations in $\Rep(Q,V)$.

To algebraically enumerate the eventually constant representations of $Q$ on $V$, we construct a directed graph $G(Q)$ on $\overline V$ and identify each representation of $Q$ on $V$ with a directed spanning subgraph. After assigning each directed subgraph a monomial weight in a polynomial ring of commuting variables, we define the generating function $P(\mathbf x)$ as the sum of weights over all graphs corresponding to an eventually constant representation. As in the $k$-cyclic case, these directed graphs become acyclic after removing edges directed along a transversal. However, unlike the cyclic case, these graphs may have higher out-degree. As a result, representations induce sets of directed acyclic graphs (DAGs) rather than sets of forests, and tree-enumeration techniques no longer apply.

We weigh collections of DAGs via recursive source removal~\cite{Gs96, Rb73}. A \defn{DAG assignment} $\mathcal D$ on the directed graph $G(Q)$ assigns each subset $U\subseteq\overline V$ a set $\mathcal D(U)$ of directed acyclic subgraphs of $G(Q)$ on the vertex set $U$. The assignment $\mathcal D$ is \defn{source-factorizable} if it satisfies certain axioms dictating how source edges can be attached to and removed from graphs in the assignment to form other graphs in the assignment. If $\mathcal D$ is source-factorizable, then the weights of the source edges that can be attached to graphs are encoded by an \defn{attaching weight matrix} $M$. The matrix $M$ is indexed over the power set $\mathcal{P}(\overline V)$ ordered by inclusion, and $M$ is strictly upper triangular. For all subsets $S\subseteq U\subseteq\overline V$, the weight of the subset of DAGs in $\mathcal D(U)$ with sink set $S\subseteq U$ is given by the entry $(I-M)^{-1}_{S,U}$. See Theorem~\ref{thm_source_factorizable_matrix_enumerator}, and Corollary~\ref{Hyperforest enumerator chain formulation-Prop} for a formulation with chains. Also, see~\cite{Rt64,St12} for more on incidence algebras. The eventually constant representations induce a source-factorizable family of DAGs, so we derive a matrix enumerator for $P(\mathbf x)$ in Section~\ref{subsec:gen_function_vertex}.

We replace the dimension vector of a linear representation with the \defn{cardinality vector}. For each subset $U\subseteq\overline V$, let $\mathbf U:=(|U\cap V_a|)_{a\in Q_0}\in\mathbb N^{Q_0}$. Then, the generating function for $\EC$ gives an expression for $|\EC|$ in terms of subsets $S\subseteq U$ of $\overline V$, and the contribution from each pair $S$, $U$ only depends on their cardinality vectors. This uniformity allows us to condense the strictly upper triangular matrix enumerator, passing from the matrix $M$ indexed on subsets $U,\,S\subseteq \overline V$ to the \defn{attaching cardinality matrix} $N$ indexed on vectors $\mathbf U,\,\mathbf S$. See Section~\ref{subsec:multinomial_notation} for an example.

To state the matrix $N$ and the enumeration result, we establish structure on the set of cardinality vectors $\mathbb N^{Q_0}$. Equip $\mathbb N^{Q_0}$ with the pointwise partial order, so that for each vector 
$\mathbf c:=(c_a)_{a\in Q_0},\,
\mathbf d:=(d_a)_{a\in Q_0}\in\mathbb N^{Q_0}$, we have $\mathbf c\leq\mathbf d$ if 
$c_a \leq d_a$ for all vertices $a\in Q_0$. For each $n\in\mathbb N$, let $\mathbf n:=(n)_{a\in Q_0}\in\mathbb N^{Q_0}$.

Let $\mathbf 0 = (0)_{a\in Q_0}$ and
$\mathbf{\overline V}=(|V_a|)_{a\in Q_0}$. 
Define the interval between $\mathbf 0$ and
$\mathbf{\overline V}$ as 
$[\mathbf 0,\mathbf{\overline V}] = \prod_{a\in Q_0} [0, |V_a|] \subseteq \mathbb N^{Q_0}$. Define exponentiation, the factorial, and $L_1$-norm on $\mathbf c,\,\mathbf d\in[\mathbf 0,\mathbf{\overline V}]$, by
\begin{equation*}
\mathbf c^\mathbf d:=\prod_{a\in Q_0} c_a^{d_a}\in\mathbb N,\qquad 
\mathbf c!:=\prod_{a\in Q_0} c_a! \in\mathbb N,\qquad |\mathbf c|:=\sum_{a\in Q_0}c_a\in\mathbb N,
\end{equation*}
respectively.
Let $\mathsf A(Q)\in\mathbb Z^{Q_0\times Q_0}$ be the \defn{adjacency matrix} of $Q$. Define the strictly upper triangular attaching cardinality matrix $N$ with entries in $[\mathbf 0,\mathbf{\overline V}]^2$, which is given by
\begin{equation}
   N_{\mathbf c,\mathbf d}:=
    \begin{cases}
        \frac{(-1)^{|\mathbf d-\mathbf c|-1}}{(\mathbf d-\mathbf c)!} 
        \mathbf c^{(\mathsf A(Q)(\mathbf d-\mathbf c))}&\text{if }\mathbf c<\mathbf d,\\
        \hspace{16mm} 0 & \text{otherwise}.
    \end{cases}
\end{equation}

In Theorem~\ref{thm:EC_cardinality_matrix_enumerator}, we derive the cardinality enumerator, restated below.
\begin{theorem}
    The cardinality of the set of eventually constant set-valued representations in $\Rep(Q,V)$ is
\begin{equation}
    |\EC|=\mathbf{\overline V}!(I-N)^{-1}_{\mathbf 1,\mathbf{\overline V}},
\end{equation}
where $I$ is the identity matrix.
\end{theorem}

In Section~\ref{section_background}, we establish our conventions for directed graphs, define source-factorizable DAG assignments, and introduce set-valued quiver representations. In Section~\ref{sec:weighted_enumeration_source_factorizable_DAGs}, we derive a recursive matrix enumerator for sets of source-factorizable DAGs. In Section~\ref{sec:eventually_constant_set-valued_representations}, we prove that eventually constant set-valued representations induce a source-factorizable DAG assignment. We then compute their generating functions in both edge and vertex variables, and evaluate the cardinality of these representations. Using chains, we compress the matrix enumerator for the cardinality calculation. Finally, in Section~\ref{sec:word_eventually_constant}, we apply our general formulas to the specific quivers, the $j$-Jordan quiver and the cyclic quiver.
In Section~\ref{sec:partition_algorithm_Jordan_quiver}, we outline an alternative algorithm for the generating function over the Jordan quiver using partitions.

\subsection*{Acknowledgments}
The authors would like to thank Mikhail Khovanov for productive discussions. 
The authors would like to thank the Department of Mathematics and the Dean's Office of Krieger School of Arts $\&$ Sciences at Johns Hopkins University for their support. 
M.S. Im would like to thank the Simons Foundation in New York, NY for a dynamic and collaborative workspace. 
The authors were partially supported by Simons Collaboration Award 994328. C. Holmes was also partially supported by NSF grant DMS-2428878.

\section{Background}
\label{section_background}
In Sections~\ref{subsection_intro_graphs} and \ref{subsection_stability_source_factor}, we give necessary background on graphs. 
All graphs under consideration are taken to be finite directed multigraphs. In Section~\ref{subsection_quivers}, we define set-valued representations of quivers. In Section~\ref{subsec:encoding_representations}, we embed the representations as directed graphs.

\subsection{Introduction to graphs}
\label{subsection_intro_graphs}
Let $G=(G_V,G_E)$ be a finite directed multigraph, where 
$G_V$ is the set of vertices and $G_E$ is the set of edges. Each edge $e\in G_E$ is directed from a single vertex $s(e)\in G_V$ called its source, to a single vertex $t(e)\in G_V$ called its target. We permit multiple parallel edges and \defn{loops} (edges from a vertex to itself).
 
An (oriented) \defn{subgraph} of $G$ is a graph $H=(H_V, H_E)$ satisfying $H_V\subseteq G_V$ and $H_E\subseteq G_E$, such that the sources and targets are inherited from $G$. We say a subgraph $H$ is a \defn{spanning subgraph} if the vertex sets are the same, $H_V=G_V$. Given a set $U$, we say a graph is \defn{on} $U$ if it has vertex set $U$.

Let $H,G$ be graphs with disjoint edge sets (but not necessarily disjoint vertex sets). Their \defn{edge-disjoint union} $H\star G$ is the graph on the combined vertex set $H_V\cup G_V$ with edge set $H_E\sqcup G_E$. Each edge of $H\star G$ retains its respective source and target from $H$ or $G$, respectively.

We define two subgraphs of $G$ induced by a subset $T\subseteq G_V$. First, the \defn{restriction} $G|_T$ is the subgraph on the vertex set $T$ consisting of all edges in $G$ whose source and target lie in $T$. Second, $G^\out_T$ is the spanning subgraph of $G$ consisting of all edges with source in $T$. 

A \defn{walk} of length $\ell \ge 1$ in $G$ is a sequence of edges $ e_\ell\cdots e_1$ where the target of each edge $e_{i-1}$ is the source of the edge $e_i$ for $i=2,\ldots,\ell$. The \defn{source} of a walk is the source $s(e_1)$ of its first edge, and the \defn{target} of a walk is the target $t(e_\ell)$ of its final edge. A walk is \defn{closed} if its source and target are the same, and $G$ is a \defn{directed acyclic graph} (DAG) if it does not admit any closed walk. Since $G$ is finite, $G$ is acyclic if and only if its walks are bounded in length.

The \defn{successors} of each vertex $u\in G_V$ are the vertices $v\in G_V$ that are the target of a walk with source $u$. In particular, if there is an edge from $u$ to $v$, then $v$ is a \defn{direct successor} of $u$.
For each subset $T\subseteq G_V$, $T$ is \defn{closed} under successors in $G$ if $T$ contains all successors of each of its vertices $u\in T$, or equivalently, if $T$ contains all direct successors of each of its vertices.
Thus, $T$ is closed under successors if and only if all edges with source in $T$ belong to $G|_T$, if and only if $G=G^\out_{G_V\setminus T}\star G|_T$. See Figure~\ref{fig_00018}.

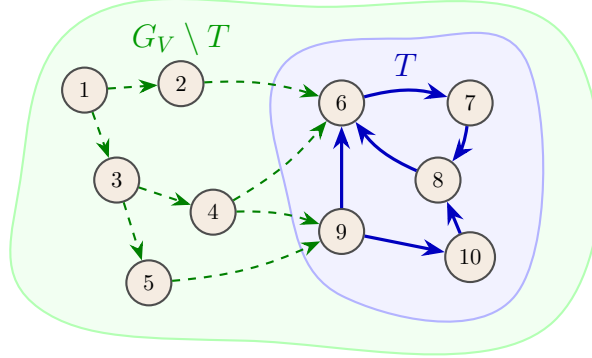
\begin{figure}[htbp]
    \centering
    \begin{tikzpicture}[
        scale=0.85, transform shape,
        domain_node/.style={circle, draw=black!70, fill=brown!15, thick, minimum size=7mm, font=\small\bfseries},
        edge_t/.style={->, >={Stealth[length=3mm]}, blue!75!black, very thick},
        edge_u/.style={->, >={Stealth[length=2.5mm]}, green!50!black, thick, dashed}
        ]

        \filldraw[fill=green!5, draw=green!30, thick] 
        plot[smooth cycle, tension=0.7] coordinates {(-0.5,2) (0.5,4.5) (4,4.7) (8,4.2) (7.8,-0.2) (3,-0.5) (-0.2,0)};
        
        \node[text=green!60!black] at (2.0, 4.2) {\Large $G_V\setminus T$};

        \filldraw[fill=blue!5, draw=blue!30, thick] 
        plot[smooth cycle, tension=0.7] coordinates {(3.6,2) (3.6,3.8) (5.5,4.2) (7.4,3.8) (7.2,0.2) (4.5,0.2)};
        
        \node[text=blue!70!black] at (5.5, 3.8) {\Large $T$};

        \node[domain_node] (u1) at (0.5, 3.4) {$1$};
        \node[domain_node] (u2) at (2.0, 3.5) {$2$};
        \node[domain_node] (u3) at (1.0, 2.0) {$3$};
        \node[domain_node] (u4) at (2.5, 1.5) {$4$};
        \node[domain_node] (u5) at (1.5, 0.4) {$5$};

        \node[domain_node] (t6) at (4.5, 3.2) {$6$};
        \node[domain_node] (t7) at (6.5, 3.2) {$7$};
        \node[domain_node] (t8) at (6.0, 2.0) {$8$};
        \node[domain_node] (t9) at (4.5, 1.2) {$9$};
        \node[domain_node] (t10) at (6.5, 0.8) {$10$};

        \draw[edge_u] (u1) to (u2);
        \draw[edge_u] (u1) to (u3);
        \draw[edge_u] (u3) to (u4);
        \draw[edge_u] (u3) to (u5);
        \draw[edge_u] (u2) to[bend left=10] (t6);
        \draw[edge_u] (u4) to[bend right=10] (t6);
        \draw[edge_u] (u4) to[bend left=10] (t9);
        \draw[edge_u] (u5) to[bend right=10] (t9);

        \draw[edge_t] (t6) to[bend left=15] (t7);
        \draw[edge_t] (t7) to[bend left=15] (t8);
        \draw[edge_t] (t8) to[bend left=15] (t6);
        \draw[edge_t] (t9) to (t10);
        \draw[edge_t] (t10) to (t8);
        \draw[edge_t] (t9) to (t6);
        
    \end{tikzpicture}
    \caption{The nodes represent the vertices of a graph $G$ and the arrows represent its arrows. The subset $T\subseteq G_V$ is closed under successors. The dashed green edges belong to the subgraph $G^{\out}_{G_V\setminus T}$ on $G_V$, and the thick blue arrows belong to $G|_T$ on $T$. Note that $G=G^{\out}_{G_V\setminus T}\star G|_T$.}
    \label{fig_00018}
\end{figure}

The \defn{out-degree} of a vertex $u\in G_V$ is the number of edges of $G$ with source $u$. The \defn{in-degree} is defined dually. Let $\Src(G)\subseteq G_V$ denote the set of \defn{source vertices}, i.e., vertices with in-degree $0$, and let $\Snk(G)\subseteq G_V$ denote the set of \defn{sink vertices}, i.e., vertices with out-degree $0$. Then, let $\Nss(G) = \Src(G)\setminus\Snk(G)$ denote the set of \defn{non-sink sources}.

When algebraically enumerating graphs, we fix an ambient graph $G$ and work in the polynomial rings of its commuting edge and vertex variables,
\[R_{G_E}:=\mathbb Z[x_e:e\in G_E],\quad R_{G_V}:=\mathbb Z[x_v:v\in G_V].\]

For each subgraph $H$ of $G$, we define the monomial weight
\[ {w}(H):=\prod_{e\in H_E}x_e\in R_{G_E}.\]
Then, we extend the function additively. For each subset $\mathcal{S}$ of subgraphs of $G$, define
\[ {w}(\mathcal S):=\sum_{H\in \mathcal S} {w}(H)\in R_{G_E}.\] 
Note that $\mathcal S$ is finite since $G$ is, hence the sum is finite.

If $H$ and $H'$ are subgraphs of $G$ with disjoint edge sets, then $H\star H'$ is also a subgraph of $G$. By disjointness, the weight factors as
\[ {w}(H\star H')=\prod_{e\in H_E\sqcup H'_E}x_e= {w}(H) {w}(H')\in R_{G_E}.\]

We pass from edge to vertex variables by the homomorphism $R_{G_E}\to R_{G_V}$ induced by the target specialization $x_e\mapsto x_{t (e)}$, where $t(e)$ is the target vertex of $e$. Under this specialization, the monomial weight of the graph $H$ becomes
\[w(H)=\prod_{e\in H_E}x_{t(e)}\in R_{G_V}.\]
Equivalently, the exponent of each variable $x_v$ is the in-degree of the vertex $v\in G_V$ in $H$.

\subsection{Source-factorizability and attaching graphs}
\label{subsection_stability_source_factor}

For this section, fix a finite ambient graph $G$. For each vertex subset $U\subseteq G_V$, let $\DAG(U)$ denote the set of directed acyclic subgraphs of $G$ on the vertex set $U$. Ultimately, we need an enumerator for appropriate subcollections of $\{\DAG(U)\}_{U\subseteq G_V}$, so we fix a \defn{DAG assignment} on $G$, a function $\mathcal D$ that maps each subset $U\subseteq G_V$ to a subset $\mathcal D(U)\subseteq\DAG(U)$.

We now introduce several auxiliary sets of graphs that encode how source edges attach to DAGs in the assignment. With these sets, we can then state uniformity axioms for $\mathcal D$ and define a matrix enumerator for sets of DAGs given by assignments that satisfy these axioms.

For subsets $T\subseteq U\subseteq G_V$, let $\mathcal B(T,U)$ denote the subset of DAGs in $\mathcal D(U)$ for which all vertices in $T$ have outgoing edges but no incoming edges
\begin{equation}
    \mathcal B(T,U):=\{H\in\mathcal D(U):T\subseteq\Nss(H)\}.
\end{equation}
We analyze how these outgoing edges in $T$ can be removed from graphs in $\mathcal B(T,U)$. 

For subsets $T\subseteq U\subseteq G_V$ and each DAG $H\in\mathcal D(U\setminus T)$, define the set of \defn{attaching graphs} $\mathcal A_{H}(T,U)\subseteq\DAG(U)$; let $H'\in\mathcal A_{H}(T,U)$
if all edges of $H'$ have source in $T$ and target in $U\setminus T$, and $H'$ attaches to $H$ to form $H'\star H$ in $\mathcal B(T,U)$.

\begin{figure}[htbp]
    \centering
\begin{tikzpicture}[
	scale=0.8, transform shape,
	domain_node/.style={circle, draw=black!70, fill=brown!15, thick, minimum size=7mm, font=\small\bfseries},
	edge_t/.style={->, >={Stealth[length=3mm]}, blue!75!black, very thick}, 
	edge_u/.style={->, >={Stealth[length=2.5mm]}, green!50!black, thick, dashed} 
	]
	
	\filldraw[fill=green!5, draw=green!30, thick] 
	plot[smooth cycle, tension=0.7] coordinates {(0, 4) (4.5, 4) (5, 1) (4.2, -1.8) (0.8, -1.8) (-0.5, 1)};
	
	\node[text=green!60!black] at (2.5, -1.7) {\Large $[10] \setminus [4]$};
	
	\node[domain_node] (u5) at (3.25, 3.2) {$5$};
	\node[domain_node] (u6) at (2.25, 2.0) {$6$};
	\node[domain_node] (u7) at (4.25, 2.0) {$7$};
	\node[domain_node] (u8) at (3, 0.5) {$8$};
	\node[domain_node] (u9) at (1.5, -1.0) {$9$};  
	\node[domain_node] (u10) at (3.5, -1.0) {$10$}; 
	
	\filldraw[fill=blue!5, draw=blue!30, thick] 
	plot[smooth cycle, tension=0.7, thick] coordinates {(-2.5, 4.5) (1, 4.5) (1.5, 2.5) (0.5, -0.5) (-1.5, -0.5) (-3, 2)};
	
	\node[text=blue!70!black] at (-0.5, 4.1) {\Large $[4]$};
	
	\node[domain_node] (t1) at (-1.5, 3.3) {$1$};
	\node[domain_node] (t2) at (0, 2.6) {$2$};
	\node[domain_node] (t3) at (-1, 1.8) {$3$};
	\node[domain_node] (t4) at (-0.5, 0.3) {$4$};
	
	\draw[edge_t] (t1) to[bend left=10] (u5);
	\draw[edge_t] (t2) to[bend left=10] (u6);
	\draw[edge_t] (t3) to[bend left=10] (u8);
	\draw[edge_t] (t4) to[bend right=10] (u9);
	
	\draw[edge_u] (u5) to (u7);
	\draw[edge_u] (u6) to (u8);
	\draw[edge_u] (u7) to (u8);
	\draw[edge_u] (u8) to (u9);
	
\end{tikzpicture}
    \caption{The dashed green arrows represent an in-forest $H\in\mathcal{D}^{(1)}([10] \setminus [4])$. The thick blue arrows represent an attaching graph $H'\in\mathcal{A}^{(1)}_H([4], [10])$. Combining their edges forms an in-forest $H'\star H\in\mathcal B^{(1)}([4], [10])$.}
    \label{fig_00016}
\end{figure}
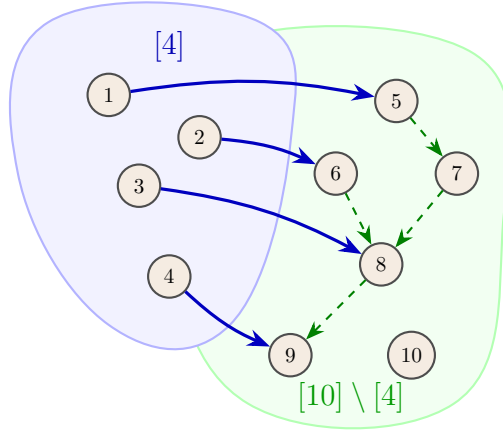

\begin{example}\label{ex:DAG_assignment_setup_1}
    For each $n\in\mathbb N$, let $[n] := \{1,\ldots,n\}$, and let $G$ be the complete directed graph on the vertex set $[10]$. We consider two DAG assignments on $G$.
    
    First, the in-forests furnish a DAG assignment on $G$. For each subset $U\subseteq [10]$, let $\mathcal D^{(1)}(U)\subseteq\DAG(U)$ be the set of directed in-forests on $U$ (DAGs where every vertex has an out-degree of at most $1$). An element $H \in \mathcal B^{(1)}([4],[10])$ is an in-forest on $[10]$ where the vertices $1, 2, 3, 4$ are non-sink sources. Thus, each vertex in $[4]$ has no incoming edges and exactly one outgoing edge with its target in $[10]\setminus[4]$, making them \defn{leaves}.

    Forests are closed under the addition of leaves. More precisely, consider any forest $H \in \mathcal D^{(1)}([10]\setminus [4])$ and any graph $H'$ on $[10]$ with a single directed edge from each vertex in $[4]$ to a vertex in $[10]\setminus[4]$. The edge-disjoint union $H' \star H$ is acyclic and its out-degree remains bounded by $1$, so $H' \star H \in \mathcal B^{(1)}([4],[10])$. Thus, $H'\in\mathcal A^{(1)}_H([4],[10])$. Furthermore, all attaching graphs in $\mathcal A^{(1)}_H([4],[10])$ are of this form. See Figure~\ref{fig_00016}.

    For the second DAG assignment on $G$, we restrict walk-length. For each subset $U\subseteq[10]$, let $\mathcal D^{(2)}(U)$ be the set of DAGs on $U$ whose walks have a maximum length of $2$. Graphs in $\mathcal B^{(2)}([4],[10])$ have at least one outgoing edge from each vertex in $[4]$ with target in $[10]\setminus[4]$, and no incoming edges in $[4]$.
    Let $H\in\mathcal D^{(2)}([10]\setminus[4])$ be a DAG with a single walk of length $2$, with source vertex $5$. Then, $H'$ is an attaching graph in $\mathcal A^{(2)}_H([4],[10])$ if its edges have source in $[4]$ and target in $([10]\setminus[4])\backslash\{5\}=[10]\backslash[5]$. See Figure~\ref{fig_00019}.
\end{example}

\begin{figure}[htbp]
    \centering
\begin{tikzpicture}[
	scale=0.8, transform shape,
	domain_node/.style={circle, draw=black!70, fill=brown!15, thick, minimum size=7mm, font=\small\bfseries},
	edge_b/.style={->, >={Stealth[length=3mm]}, blue!75!black, very thick}, 
	edge_r/.style={->, >={Stealth[length=3mm]}, red!75!black, ultra thick, dotted}, 
	edge_u/.style={->, >={Stealth[length=2.5mm]}, green!50!black, thick, dashed} 
	]
	
	\filldraw[fill=green!5, draw=green!30, thick] 
	plot[smooth cycle, tension=0.7] coordinates {(0, 4) (4.5, 4) (5, 1) (4.2, -1.8) (0.8, -1.8) (-0.5, 1)};
	
	\node[text=green!60!black] at (2.75, -1.65) {\Large $[10] \setminus [4]$};
	
	\node[domain_node] (u5) at (4.25, 3.2) {$5$};
	\node[domain_node] (u6) at (3.25, 2.25) {$6$};
	\node[domain_node] (u7) at (4.5, 2.0) {$7$};
	\node[domain_node] (u8) at (4.0, 0.5) {$8$};
	\node[domain_node] (u9) at (1.5, -1.2) {$9$}; 
	\node[domain_node] (u10) at (3.5, -0.75) {$10$}; 
	
	\filldraw[fill=blue!5, draw=blue!30, thick] 
	plot[smooth cycle, tension=0.7, thick] coordinates {(-2.5, 4.5) (1, 4.5) (1.5, 2.5) (0.5, -0.5) (-1.5, -0.5) (-3, 2)};
	
	\node[text=blue!70!black] at (-0.5, 4.1) {\Large $ [4]$};
	
	\node[domain_node] (t1) at (-1.5, 3.3) {$1$};
	\node[domain_node] (t2) at (0.6, 3.05) {$2$};
	\node[domain_node] (t3) at (-1, 1.8) {$3$};
	\node[domain_node] (t4) at (-0.9, 0.3) {$4$};
	
	\draw[edge_u] (u5) to (u6);
	\draw[edge_u] (u6) to (u7);
	\draw[edge_u] (u8) to (u10);
	\draw[edge_u] (u9) to (u10);
	
	\draw[edge_b] (t1) to[bend right=10] (u6);
	\draw[edge_b] (t2) to[bend left =13] (u8);
	\draw[edge_b] (t3) to[bend right=10] (u9);
    \draw[edge_b] (t3) to[bend right=15] (u8);
	\draw[edge_b] (t4) to[bend right=15] (u9);
	
	\draw[edge_r] (t1) to[bend left=10] (u5);
	\draw[edge_r] (t2) to[bend left=15] (u6);
	\draw[edge_r] (t3) to[bend left=10] (u8);
	\draw[edge_r] (t4) to[bend left=10] (u10);
	
\end{tikzpicture}
    \caption{The dashed green arrows represent a graph $H\in\mathcal D^{(2)}([10]\setminus[4])$ with a single walk of length $2$ starting at vertex $5$. The thick blue arrows represent an attaching graph in $\mathcal A^{(2)}_H([4],[10])$, and the dotted red arrows represent a graph that is not in $\mathcal A^{(2)}_H([4],[10])$.}
    \label{fig_00019}
\end{figure}
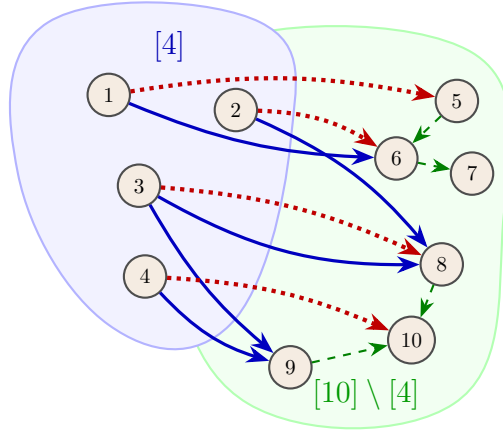

The assignment $\mathcal D$ on $G$ is \defn{source-factorizable} if for any subsets $T\subseteq U\subseteq G_V$,
    \begin{enumerate}[(i)]
    \item \label{Leaf-deletion condition}
    for each $H\in\mathcal B(T,U)$, the restricted subgraph $H|_{U\setminus T}$ is in $\mathcal D({U\setminus T})$,
    \item \label{DAG independence condition}
    all DAGs $H\in\mathcal D({U\setminus T})$ induce the same attaching set $\mathcal A_{H}(T,U)$, and 
    \item \label{Nontriviality_condition}
    each collection $\mathcal D(U)$ contains the edgeless graph on $U$.
    \end{enumerate}
Condition~\ref{Leaf-deletion condition} states that $\bigcup_{U\subseteq G_V}\mathcal D(U)$ is closed under the removal of non-sink source vertices and their outgoing edges. Condition~\ref{Nontriviality_condition} asserts that the sets $\mathcal D(U)$ are nonempty. Condition~\ref{DAG independence condition} gives uniformity of the attaching sets. Since the set $\mathcal D(U\setminus T)$ is nonempty by Condition~\ref{Nontriviality_condition}, we can define $\mathcal A(T,U)$ as the unique set $\mathcal A_H(T,U)$ induced by any $H\in\mathcal D(U\setminus T)$.

\begin{example}
    We return to the DAG assignments on the complete directed graph $G$ on $[10]$.

    The in-forest assignment $\mathcal D^{(1)}$ is source-factorizable. For each in-forest $H\in\mathcal B^{(1)}([4],[10])$, the vertices in $[4]$ are leaves and forests are closed under leaf removal, so the restriction $H|_{[10]\setminus [4]}$ is a forest on $[10]\setminus [4]$. Also, for each in-forest $H\in\mathcal D^{(1)}([10]\setminus[4])$, the attaching set $\mathcal A_H^{(1)}([4],[10])$ consists of all graphs on $[10]$ with a single edge from each vertex of $[4]$ to a vertex in $[10]\setminus [4]$. These results for $[4],\,[10]$ generalize to all subsets $T\subseteq U\subseteq [10]$, so $\mathcal D^{(1)}$ satisfies Conditions~\ref{Leaf-deletion condition} and~\ref{DAG independence condition}.  Condition~\ref{Nontriviality_condition} holds since the edgeless graph on each set is an in-forest.

    In contrast, the walk-length assignment $\mathcal D^{(2)}$ is not source-factorizable. Let $H,\,H'\in\mathcal D^{(2)}([10]\setminus [4])$, and only $H$ has a walk of length $2$ starting at vertex $5$. Then, attaching graphs in $\mathcal A^{(2)}_{H'}([4],[10])$ can have edges targeting vertex $5$, whereas attaching graphs in $\mathcal A^{(2)}_H([4],[10])$ cannot.
\end{example}

Fix a source-factorizable DAG assignment $\mathcal D$ on $G$. Then, for all subsets $T\subseteq U\subseteq G_V$, define the \defn{attaching weight},
\begin{equation}
     {A}(T,U):= {w}(\mathcal A(T,U))\in R_{G_E}.
\end{equation}
With these weights, we construct the \defn{attaching weight matrix} $ {M}$. Define the square matrix indexed by the subsets of $G_V$ with entries in $R_{G_E}$ given by
\begin{equation}\label{eq_W}
	 {M}_{T,U}:=
	\begin{cases}
	(-1)^{|U\setminus T|-1} {A}(U\setminus T,{U})&\text{if } T\subsetneq U,\\
	\hspace{18mm} 0&\text{otherwise.}
	\end{cases}
\end{equation}
Since ${M}$ is strictly upper triangular, it is nilpotent. So, $(I- {M})^{-1}$ expands as a finite geometric sum in $ {M}$, with $I$ denoting the identity matrix.

Finally, we define the sink-restricted subsets of our constructions above.
For all subsets $T\subseteq U\subseteq G_V$ and $S\subseteq G_V$, let $\mathcal D_S(U)$ and $\mathcal B_S(T,U)$ denote subsets of DAGs in $\mathcal D(U)$ and $\mathcal B(T,U)$ that have sink set $S$,
\begin{equation}
\label{eqn_B_S_TU}
\mathcal D_S(U):=\{H\in\mathcal D(U):\Snk(H)=S\},\quad \mathcal B_S(T,U):= \{H\in\mathcal B(T,U):\Snk(H)=S\}.
\end{equation}
Note $\mathcal D_S(U)$ is empty if $S\not\subseteq U$, and $\mathcal B_S(T,U)$ is empty if $S\not\subseteq U\setminus T$. The attaching weight matrix gives an enumeration for $\mathcal D_S(U)$ in Theorem~\ref{thm_source_factorizable_matrix_enumerator}.

\begin{example}
    Consider the source-factorizable DAG assignment $\mathcal D^{(1)}$ of in-forests on $G$. For all subsets $T\subseteq U\subseteq[10]$, the attaching graphs $\mathcal A(T,U)$ are precisely the graphs that assign each vertex in $T$ exactly one edge with target in $U\setminus T$. These graphs identify with the functions from $T$ to $U\setminus T$, so it can be shown that the attaching weight factors into a product over the source vertices:
    \begin{equation}
        A(T,U) = \prod_{u\in T}\sum_{v\in U\setminus T}x_{u,v} \in R_{G_E}.
    \end{equation}
    Under the target variable specialization $x_{u,v} \mapsto x_v$, the inner sum simplifies to
    \begin{equation}
        A(T,U) = e_1(U\setminus T)^{|T|} \in R_{G_V},
    \end{equation}
    where $e_1$ is the elementary symmetric polynomial of degree $1$, $e_1(U\setminus T) := \sum_{v\in U\setminus T} x_v$.
    Under the target variable specialization, the attaching weight matrix $M$ has entries
    \begin{equation}
        M_{T,U}=
        \begin{cases}
            (-1)^{|U\setminus T|-1} e_1(T)^{|U\setminus T|} &\text{if }T\subsetneq U,\\
            \hspace{18mm} 0 &\text{otherwise.}
        \end{cases}
    \end{equation}
    Also, the sink vertices of an in-forest are precisely its roots, so for all subsets $S\subseteq U\subseteq [10]$, $\mathcal D^{(1)}_S(U)$ is the subset of forests on $U$ with root set $S$.
\end{example}

\subsection{Partially ordered sets and chains}
\label{subsec_partially_ordered_chains}

Let $P$ be a partially ordered set. A strictly increasing \defn{chain} of length $d\in\mathbb N$ is a $(d+1)$-tuple $\Gamma:=(\Gamma_0,\Gamma_1,\ldots,\Gamma_d)$ in $P$ such that $\Gamma_{i-1}<\Gamma_{i}$ for each $i=1,\ldots,d$. We say $\Gamma$ is a chain from its first element $\Gamma_0$ to its last element $\Gamma_d$, and we denote its length by $L(\Gamma)=d$.

For each pair of elements $b,\,c\in P$, let $\mathcal C(b,c)$ denote the set of strictly increasing chains from $b$ to $c$. When working with multiple partially ordered sets, we specify $P$ with the subscript $\mathcal C_P(b,c)$.

\begin{example}
    Consider the partially ordered power set $\mathcal P([10])$. Then $\Gamma=([2],[3],[3]\cup\{5,7,9\},[9])$ is a chain from $[2]$ to $[9]$, and $L(\Gamma)=3$.
\end{example}

\subsection{Set-valued quiver representations}
\label{subsection_quivers}

In this section, fix a finite \defn{quiver} $Q = (Q_0,Q_1)$, which is a finite directed multigraph, where $Q_0$ is the set of vertices and $Q_1$ is the set of arrows. We also fix a $Q_0$-indexed tuple $V:=(V_a)_{a \in Q_0}$ of nonempty, finite, disjoint sets and abbreviate $\overline V:=\coprod_{a\in Q_0}V_a$. For each arrow $\alpha\in Q_1$, let $s(\alpha)$ and $t(\alpha)$ denote its \defn{source} and \defn{target} vertices, respectively, written $\alpha:s(\alpha)\to t(\alpha)$. For each subset $U\subseteq \overline V$ and each vertex $a\in Q_0$, let $U_a:=U\cap V_a$.

A \defn{path} of length $\ell \ge 1$ in $Q$ is a sequence of arrows $p := \alpha_\ell \cdots \alpha_1$ where $t(\alpha_{i-1}) = s(\alpha_i)$ for each $i=2,\ldots,\ell$. 
As with directed graphs, the source of a path is the source of its first arrow and the target of a path is the target of its last arrow. A path of length $\geq 1$ is an \defn{oriented cycle} if it has the same source and target. A \defn{successor} of a vertex $a\in Q_0$ is any vertex $b\in Q_0$ for which there is a path $p$ from $a$ to $b$, and a subset $X\subseteq Q_0$ is \defn{closed under successors} if $X$ contains the successors of each of its vertices.
A vertex is a \defn{sink} if it has no successors.

Throughout this paper, we assume $Q$ has no sinks and that every vertex is the target of a path of length $\geq|Q_0|$. Since $Q_0$ is finite, this is equivalent to assuming every vertex is the target of arbitrarily long paths, or rather every vertex is a successor of a vertex belonging to a directed cycle.

The \defn{adjacency matrix} $\mathsf{A}(Q)\in \mathbb{Z}^{Q_0 \times Q_0}$ is defined such that $\mathsf{A}(Q)_{b,a}$ is the number of arrows from $a$ to $b$. The \defn{in-degree} of a vertex $b\in Q_0$ is the number of arrows with target $b$, denoted $d_{in}^{Q_0}(b)\in\mathbb N$.

In the literature, see e.g.~\cite{DW05,DW17}, a set-valued representation $W$ of a quiver $Q$ consists of a collection $\{W_a\}_{a\in Q_0}$ of finite sets and a collection $\{ W_\alpha:W_{s(\alpha)}\to W_{t(\alpha)}\}_{\alpha\in Q_1}$ of maps. But in this paper, a \defn{set-valued representation} $W$ of $Q$ is a $Q_0$-indexed tuple $(W_a)_{a\in Q_0}$ of finite sets and a $Q_1$-indexed tuple $(W_\alpha:W_{s(\alpha)}\to W_{t(\alpha)})_{\alpha\in Q_1}$ of maps between these sets.

Alternatively, let $\mathcal Q$ be the free category generated by $Q$, called the \defn{path category}. Its objects are the vertices $a\in Q_0$, its morphisms are paths in $Q$, and its composition is path concatenation. On each vertex $a\in Q_0$, we also specify a $0$-length path on $a$ to serve as the identity. Then, set-valued representations correspond precisely to functors from $\mathcal Q$ to the category of finite sets. For any subset $U\subseteq \overline V$, let $\Rep(Q,(U_a)_{a\in Q_0})$ denote the set of \defn{set-valued representations on $U$}, which are the set-valued representations $W$ of $Q$ such that $W_a=U_a$ for each $a\in Q_0$. Hereafter, we may omit the prefix ``set-valued".

\begin{example}\label{ex:quiver_rep_3cycle}
    Let $Q$ be the $3$-cycle quiver with vertex set $Q_0 = \{1, 2, 3\}$ and arrows 
    \[1\xrightarrow{\alpha_1}2\xrightarrow{\alpha_2}3\xrightarrow{\alpha_3}1.\]
   Then, $Q$ has no sinks and every vertex is the successor of an element in a directed cycle.
    
    Define the disjoint sets $V_1 = \{x_1, x_2, x_3\}$, $V_2 = \{y_1, y_2, y_3\}$, and $V_3 = \{z_1, z_2, z_3\}$. A representation $W$ on $V:=(V_1,V_2,V_3)$ is determined by its $3$ maps on arrows,
\[V_1\xrightarrow{W_{\alpha_1}}V_2\xrightarrow{W_{\alpha_2}}V_3\xrightarrow {W_{\alpha_3}}V_1.\]
Morphisms in the path category $\mathcal{Q}$ are paths in $Q$. Paths $p:=\alpha_{2}\alpha_{1}\alpha_3$ and $p':=\alpha_2\alpha_1$ concatenate to form the path $pp'=\alpha_2\alpha_1\alpha_3\alpha_2\alpha_1$. The path $pp'$ evaluates to the composite map $W_{pp'} = W_{\alpha_2}\circ W_{\alpha_1}\circ W_{\alpha_3}\circ W_{\alpha_2}\circ W_{\alpha_1}: V_1 \to V_3$.
\end{example}

Let $U\subseteq \overline V$. Then a representation $W\in\Rep(Q,(U_a)_{a\in Q_0})$ \defn{stabilizes} in a subset $S\subseteq U$ if $S$ is closed under the maps $(W_\alpha)_{\alpha\in Q_1}$, and for all sufficiently long paths $p$ in $Q$, the function $W_p$ maps into $S$. More precisely, $W$ stabilizes in $S$ if for all arrows $\alpha\in Q_1$, $W_\alpha$ maps $S_{s(\alpha)}$ into $S_{t(\alpha)}$, and if for all paths $p:=\alpha_\ell\cdots \alpha_1$ in $Q$ longer than a fixed length $\ell_0\in\mathbb N$, $W_p$ maps $U_{s(\alpha_1)}$ into $S_{t(\alpha_\ell)}$. Let $\EC_S(U)$ denote the set of representations on $U$ that stabilize in $S\subseteq U$. Note that $\ell_0$ can be taken as $|U|-|S|$. To see this, consider any representation $W$, any path $p:=\alpha_\ell\cdots \alpha_1$, and any vertex $u\in U_{s(\alpha_1)}$.
If the sequence of vertices $(u,W_{\alpha_{1}}(u),\ldots,W_{\alpha_{\ell}\cdots\alpha_1}(u))$ repeats before entering $S$, then there exists a nontrivial path $p'$ and vertex $v\in U\setminus S$ for which $W_{p'}(v)=v$, so $W\not\in\EC_S(U)$. Thus, if $W\in\EC_S(U)$ and $\ell>|U|-|S|$, then $W_p(u)\in S$.

A \defn{transversal} of $V$ is a vertex subset $C\subseteq \overline V$ such that for each vertex $a\in Q_0$, $C_a\subseteq V_a$ is a singleton. We denote the unique element of each set $C_a$ by $c_a$, and denote the set of transversals as $\mathcal T(V)$. We say that a representation on $V$ is \defn{eventually constant} if it stabilizes in a transversal, and we let $\EC$ denote the set of such representations on $\overline V$.
Given a vertex subset $U\subseteq \overline V$, we say that a representation $W\in\Rep(Q,(U_a)_{a\in Q_0})$ is \defn{eventually constant on $U$} if the representation stabilizes in a transversal $C\in\mathcal T(V)$ contained in $U$.

Observe that the closure condition is redundant for transversals $C\in\mathcal T(V)$. If $W\in\Rep(Q,V)$ and for all sufficiently long paths $p$ in $Q$ the function $W_p$ maps into $C$, then $C$ is closed under the functions $(W_\alpha)_{\alpha\in Q_1}$; for each arrow $\alpha:a\to b$, there are sufficiently long paths $p$ in $Q$ with target $a\in Q_0$, so $\alpha p$ is a sufficiently long path. Thus, $W_p$ maps into $\{c_a\}$, and $W_\alpha\circ W_p$ maps into $\{c_b\}$, so $W_\alpha(c_a)=c_b$.

A quiver $Q'$ is a \defn{subquiver} of $Q$ if all vertices and arrows in $Q'$ belong to $Q$, and the source and target of each arrow in $Q'$ is the same as in $Q$. If $Q'$ is a subquiver of $Q$, then each representation $W\in\Rep(Q,(U_a)_{a\in Q_0})$ restricts to a representation $W|_{Q'}\in\Rep(Q',(U_a)_{a\in Q'_0})$.

If $W\in\Rep(Q,V)$ and $S\subseteq\overline V$, a \defn{subrepresentation} of $W$ is a representation $W'\in\Rep(Q,(S_a)_{a\in Q_0})$ such that each arrow $\alpha\in Q_1$, $W_\alpha':S_{s(\alpha)}\to S_{t(\alpha)}$ is the restriction of the map $W_\alpha:V_{s(\alpha)}\to V_{t(\alpha)}$. Note that if $S$ is closed under the maps $(W_\alpha)_{\alpha\in Q_1}$, then $W$ induces a subrepresentation $W'\in\Rep(Q,(S_a)_{a\in Q_0})$.

\subsection{Encoding quiver representations as graphs}
\label{subsec:encoding_representations}

Fix a finite quiver $Q$ and a $Q_0$-indexed tuple $(V_a)_{a\in Q_0}$ of finite disjoint nonempty sets. To algebraically enumerate the eventually constant representations of $Q$, we first define an ambient graph $G(Q)$. Then, we identify the representations with subgraphs of $G(Q)$ so that we can weigh them. 

Let $G(Q)$ be a directed multigraph on the vertex set $\overline V = \coprod_{a\in Q_0}V_a$. For each arrow $\alpha:a\to b$ in $Q_1$, let $G(Q)$ contain a set of $\alpha$-indexed edges from each vertex in $V_a$ to each vertex in $V_b$,
\[ \left\{ (u,v)^{(\alpha)} : u\in V_a,\,v\in V_b \right\}, \]
where each edge $(u,v)^{(\alpha)}$ is directed from $u$ to $v$. For example, if $Q$ has a single vertex $a$ and a single loop at $a$, then $G(Q)$ is the complete graph on $V_a$. See Figure~\ref{fig_00021}.

\begin{figure}[htbp]
    \centering
    \begin{tikzpicture}[scale=1, transform shape, baseline=(current bounding box.center),
        quiver_node/.style={circle, draw=black!70, fill=brown!15, very thick, minimum size=8.5mm, inner sep=0.5pt, font=\large},
        edge_b/.style={->, >={Stealth[length=2.5mm]}, blue!75!black, very thick},
        edge_g/.style={->, >={Stealth[length=2.5mm]}, green!50!black, very thick, dashed},
        edge_r/.style={->, >={Stealth[length=2.5mm]}, red!75!black, very thick, dotted}
        ]
        
        \node[quiver_node] (q1) at (0,0) {$1$};
        \node[quiver_node] (q2) at (2.5,0) {$2$};

        \draw[edge_b] (q1) edge[bend left=30] node[midway, above=1mm] {$\alpha$} (q2);
        \draw[edge_g] (q2) edge[bend left=30] node[midway, below=1mm] {$\beta$} (q1);
        \draw[edge_r] (q2) edge[loop right, in=330, out=30, looseness=8, min distance=12mm] node[midway, right=1mm] {$\gamma$} (q2);
    \end{tikzpicture}
    \hspace{1.5cm}
    \begin{tikzpicture}[scale=0.9, transform shape, baseline=(current bounding box.center),
        quiver_node/.style={circle, draw=black!70, fill=brown!15, very thick, minimum size=8.5mm, inner sep=0.5pt, font=\large},
        edge_b/.style={->, >={Stealth[length=2.5mm]}, blue!75!black, very thick},
        edge_g/.style={->, >={Stealth[length=2.5mm]}, green!50!black, very thick, dashed},
        edge_r/.style={->, >={Stealth[length=2.5mm]}, red!75!black, very thick, dotted}
        ]

        \node[quiver_node] (x1) at (0, 1.5) {$x_1$};
        \node[quiver_node] (x2) at (0, 0) {$x_2$};
        \node[quiver_node] (y1) at (4, 0.90) {$y_1$};
        \node[quiver_node] (x3) at (0, -1.5) {$x_3$};
        \node[quiver_node] (y2) at (4, -0.90) {$y_2$};

        \draw[edge_b] (x1) edge[bend left=16] (y1);
        \draw[edge_b] (x1) edge[bend left=10] (y2);
        \draw[edge_b] (x2) edge[bend left=4] (y1);
        \draw[edge_b] (x2) edge[bend left=4] (y2);
        \draw[edge_b] (x3) edge[bend left=4] (y1);
        \draw[edge_b] (x3) edge[bend left=4] (y2);

        \draw[edge_g] (y1) edge[bend right=3] (x1);
        \draw[edge_g] (y1) edge[bend left=6] (x2);
        \draw[edge_g] (y1) edge[bend left=6] (x3);
        \draw[edge_g] (y2) edge[bend left=5] (x1);
        \draw[edge_g] (y2) edge[bend left=8] (x2);
        \draw[edge_g] (y2) edge[bend left=15] (x3);

        \draw[edge_r] (y1) edge[loop right, in=330, out=30, looseness=8, min distance=12mm] (y1);
        \draw[edge_r] (y2) edge[loop right, in=330, out=30, looseness=8, min distance=12mm] (y2);
        
        \draw[edge_r] (y1) edge[bend left=15] (y2);
        \draw[edge_r] (y2) edge[bend left=15] (y1);
    \end{tikzpicture}
    
    \caption{On the left is a quiver $Q$ with two vertices $\{1,2\}$ and arrows $\alpha:1\to 2$, $\beta:2\to 1$, and $\gamma:2\to 2$. On the right is the graph $G(Q)$ on the sets $V_1=\{x_1, x_2, x_3\}$ and $V_2 = \{y_1,y_2\}$. The solid blue, dashed green, and dotted red arrows in the quiver induce edges depicted in the same way in the graph $G(Q)$.}
    \label{fig_00021}
\end{figure}
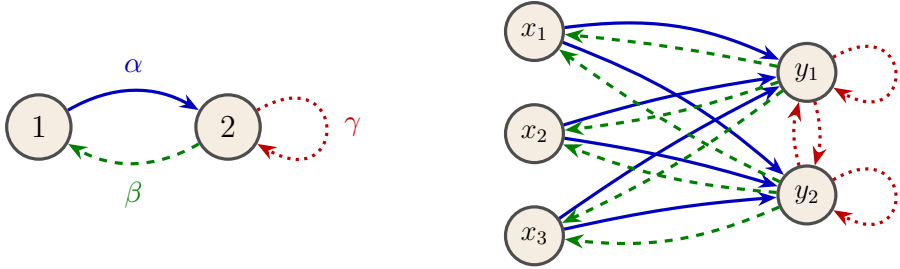

Consider any subset $U\subseteq\overline V$ and representation $W\in\Rep(Q,(U_a)_{a\in Q_0})$. Associate $W$ with the subgraph $G(W)$ of $G(Q)$ on the vertex set $U$ with the disjoint union of edges
\[ \coprod_{\alpha\in Q_1}\left\{ (u,W_\alpha(u))^{(\alpha)} : u\in U_a \right\}.\]
Essentially, $G(W)$ encodes each function $W_\alpha:W_{s(\alpha)}\to W_{t(\alpha)}$ by the edge set, $\{(u,W_\alpha(u))^{(\alpha)} : u\in U_a\}$, which corresponds to its standard set-theoretic graph labeled by $\alpha$. Let $\mathcal R(U)$ denote the set of graphs induced by $\Rep(Q,(U_a)_{a\in Q_0})$,
\[\mathcal R(U):=\{G(W) : W\in\Rep(Q,(U_a)_{a\in Q_0})\}.\]

Fix subsets $S\subseteq U\subseteq\overline V$ and a representation $W\in\Rep(Q,(U_a)_{a\in Q_0})$. Then $W$ shares properties with its graph $G(W)$. The set $S$ is closed under the maps $(W_\alpha)_{\alpha\in Q_1}$ if and only if $S$ is closed under successors in $G(W)$. Also, the graph $G(W)$ has a walk
\begin{equation}\label{walk_in_mathcal F}
    (u_1,u_2)^{(\alpha_1)},\,(u_2,u_3)^{(\alpha_2)},\,\ldots,\,(u_{\ell},u_{\ell+1})^{(\alpha_\ell)}
\end{equation}
if and only if $p := \alpha_\ell\cdots \alpha_1$ is a path in $Q$, $u_1\in U_{s(\alpha_1)}$, and each $u_{i+1}=W_{\alpha_i\cdots \alpha_1}(u_1)$. So if $W\in\EC_S(U)$, then $S$ is closed under successors in $G(W)$ and all sufficiently long walks in $G(W)$ have target in $S$. The converse holds similarly.

\begin{example}\label{Ex:3-cycle_quiver_graph}
    Recall the $3$-cycle quiver $Q$. The graph $G(Q)$ has the vertex set
    \[ \coprod_{i=1}^3 V_i = \{x_1, x_2, x_3, y_1, y_2, y_3, z_1, z_2, z_3\}. \]
    The graph $G(Q)$ has the edges $(x_q,y_r)^{(\alpha_1)}$, $(y_q,z_r)^{(\alpha_2)}$, and $(z_q,x_r)^{(\alpha_3)}$ as $q, r$ range over $\{1,2,3\}$. For each representation $W$, the graph $G(W)$ contains the edges $(x_q, W_{\alpha_1}(x_q))^{(\alpha_1)}$, $(y_q, W_{\alpha_2}(y_q))^{(\alpha_2)}$, and $(z_q, W_{\alpha_3}(z_q))^{(\alpha_3)}$ for $q \in \{1,2,3\}$. See Figure~\ref{fig_00020}.

    Transversals are subsets $C:=\{x_{i_1}, y_{i_2}, z_{i_3}\}$. A representation $W$ stabilizes in a transversal $C \subseteq \coprod_{i=1}^3 V_i$ if for all sufficiently long paths $p$ in $Q$, the map $W_p$ maps into $C$, or equivalently, if all sufficiently long walks in $G(W)$ end in $C$. This occurs if and only if the composite map $W_{\alpha_3}\circ W_{\alpha_2}\circ W_{\alpha_1}$ is eventually constant on $V_1$ to the vertex $x_{i_1}$ in the classical sense.
\end{example}

\begin{figure}[htbp]
    \centering
    \begin{tikzpicture}[scale=1, transform shape, baseline=(current bounding box.center),
        quiver_node/.style={circle, draw=black!70, fill=orange!15, very thick, minimum size=8.5mm, inner sep=0.5pt, font=\large},
        edge_f1/.style={->, >={Stealth[length=3mm]}, blue!75!black, very thick},
        edge_f2/.style={->, >={Stealth[length=2.5mm]}, green!50!black, very thick, dashed},
        edge_f3/.style={->, >={Stealth[length=3mm]}, red!75!black, very thick, dotted}
        ]

        \node[quiver_node] (q1) at (0, 1.5) {$1$};
        \node[quiver_node] (q2) at (1.5, -1) {$2$};
        \node[quiver_node] (q3) at (-1.5, -1) {$3$};

        \draw[edge_f1] (q1) edge[bend left=15] node[midway, right=1mm] {$\alpha_1$} (q2);
        \draw[edge_f2] (q2) edge[bend left=15] node[midway, below=1mm] {$\alpha_2$} (q3);
        \draw[edge_f3] (q3) edge[bend left=15] node[midway, left=1mm] {$\alpha_3$} (q1);
    \end{tikzpicture}
    \hspace{1.5cm}
    \begin{tikzpicture}[scale=0.85, transform shape, baseline=(current bounding box.center),
        >={Stealth[length=2.5mm]},
        domain_node/.style={circle, draw=black!70, fill=orange!15, very thick, minimum size=8.5mm, inner sep=0.5pt, font=\large},
        transversal_node/.style={circle, draw=MidnightBlue, fill=blue!15, very thick, minimum size=8.5mm, inner sep=0.5pt, font=\large},
        edge_f1/.style={->, >={Stealth[length=3mm]}, blue!75!black, very thick}, 
        edge_f2/.style={->, >={Stealth[length=2.5mm]}, green!50!black, very thick, dashed}, 
        edge_f3/.style={->, >={Stealth[length=3mm]}, red!75!black, very thick, dotted} 
        ]

        \node[transversal_node] (u1) at (-3, -1.5) {$x_1$};
        \node[transversal_node] (u2) at (0, 0.25) {$y_1$};
        \node[transversal_node] (u3) at (3, -1.5) {$z_1$};

        \node[domain_node] (x2) at (-3, 0) {$x_2$};
        \node[domain_node] (y2) at (0, 1.50) {$y_2$};
        \node[domain_node] (z2) at (3, 0) {$z_2$};

        \node[domain_node] (x3) at (-3, 1.5) {$x_3$};
        \node[domain_node] (y3) at (0, 2.75) {$y_3$};
        \node[domain_node] (z3) at (3, 1.5) {$z_3$};

        \draw[edge_f1] (u1) edge[bend left=15] (u2);
        \draw[edge_f1] (x2) edge[bend left=15] (u2);
        \draw[edge_f1] (x3) edge[bend left=15] (y2);

        \draw[edge_f2] (u2) edge[bend left=15] (u3);
        \draw[edge_f2] (y2) edge[bend left=10] (z2);
        \draw[edge_f2] (y3) edge[bend left=15] (z2);

        \draw[edge_f3] (u3) edge[bend left=15] (u1);
        \draw[edge_f3] (z2) edge[bend left=17] (x2);
        \draw[edge_f3] (z3) edge[bend left=15] (u1);

    \end{tikzpicture}
    
    \caption{On the left is the $3$-cycle quiver $Q$. On the right is the graph $G(W)$ of a representation $W\in\Rep(Q,V)$, with $V:=\{x_i\}_{i\leq 3}\sqcup\{y_j\}_{j\leq 3}\sqcup\{z_k\}_{k\leq 3}$. The edges induced by $W_{\alpha_1}:V_1\to V_2$ are solid blue, the edges induced by $W_{\alpha_2}:V_2\to V_3$ are dashed green, and the edges induced by $W_{\alpha_3}:V_3\to V_1$ are dotted red. The representation stabilizes in the transversal $\{x_1,y_1,z_1\}$.}
    \label{fig_00020}
\end{figure}
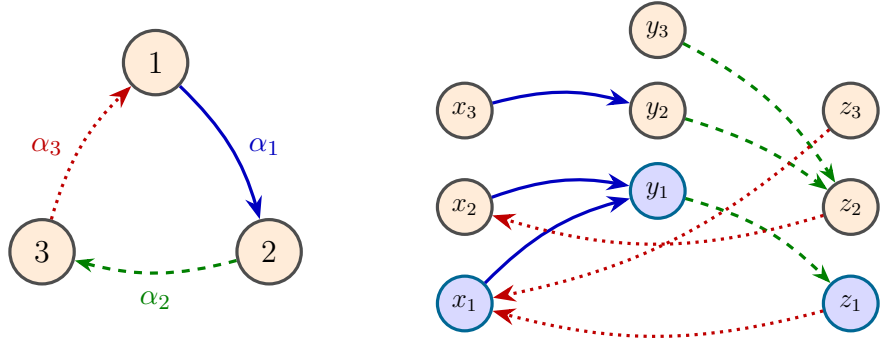

It will be convenient to expand the graph assignment $W\mapsto G(W)$ from the set of representations to slightly more general tuples of $Q_1$-indexed functions.
Given subsets $T\subseteq U\subseteq \overline V$, assign each $Q_1$-indexed tuple
\[W:=(W_\alpha:T_{s(\alpha)}\to U_{t(\alpha)})_{\alpha\in Q_1}\]
to the graph $G(W)$ on $U$ with the edge set
\[
\coprod_{\alpha\in Q_1}\{(u,W_\alpha(u))^{(\alpha)} : u\in T_{s(\alpha)}\subseteq U_{s(\alpha)}\}.
\]
For all subsets $T,S\subseteq U$, let $\mathcal R_U(T\to S)$ be the set of subgraphs $G(W)$ of $G(Q)$ on $U$ formed by the tuples $(W_\alpha:T_{s(\alpha)}\to U_{t(\alpha)})_{\alpha\in Q_1}$ such that $\im W_\alpha\subseteq S_{t(\alpha)}$ for each arrow $\alpha\in Q_1$.
We say a graph $H$ on $U$ is \defn{representable} if $H\in\mathcal R_U(T\to U)$ for some subset $T\subseteq U$; $H$ is represented by the tuple $W$ such that $G(W)=H$. Note that $\mathcal R(U)=\mathcal R_U(U\to U)$.

For instance, if $Q$ is the $3$-cycle quiver, graphs in $\mathcal{R}_U(T \to S)$ correspond to the tuples, $(W_{\alpha_1}, W_{\alpha_2}, W_{\alpha_3})$, where $W_{\alpha_1}: T_1 \to S_2$, $W_{\alpha_2}: T_2 \to S_3$, and $W_{\alpha_3}: T_3 \to S_1$. Equivalently, each graph $H\in\mathcal{R}_U(T \to S)$ contains exactly one outgoing edge from each vertex $u \in T_i$ targeting a vertex in $S_{i+1}$ (with indices modulo 3).

For clarity, we write each formal edge variable $x_{(u,v)^{(\alpha)}}$ as $x_{u,v}^{(\alpha)}$. For each vertex subset $U\subseteq \overline V$ and each graph $G(W)\in\mathcal R(U)$, the graph $G(W)$ has the edge weight
\[ {w}(G(W))=\prod_{\alpha\in Q_1}\prod_{u\in U_{s(\alpha)}} x^{(\alpha)}_{u,W_\alpha(u)}\in R_{G(Q)_E}.\]
Specialized to target variables, the weight is
\[w(G(W))=\prod_{\alpha\in Q_1}\prod_{u\in U_{s(\alpha)}} x_{W_\alpha(u)}\in R_{\overline V}.\]
Then, let $P(\mathbf x)$ be the generating function over the eventually constant representations,
\[P(\mathbf x) :=\sum_{W\in\EC}{w(G(W))}.\]
More generally, for each vertex subset $U\subseteq \overline V$, let $P_U(\mathbf x)$ denote the generating function over the eventually constant representations on $U$.

\subsection{Cardinality vectors and attaching cardinality matrices}
\label{subsec:multinomial_notation}

As in Section~\ref{subsection_quivers}, fix a finite quiver $Q = (Q_0, Q_1)$ and a $Q_0$-indexed tuple $V:=(V_a)_{a\in Q_0}$ of finite, nonempty, disjoint sets.

Equip $\mathbb N^{Q_0}$ with the pointwise order, so that 
$\mathbf c\leq\mathbf d$ if and only if $c_a \leq d_a$ for each $a\in Q_0$. Define the factorial of $\mathbf c\in\mathbb N^{Q_0}$ as the product of its component factorials,
\[
\mathbf c!:=\prod_{a\in Q_0} c_a!\in\mathbb N.
\]
We extend to multinomials. If $\mathbf c\leq \mathbf d$, or more generally if $\mathbf c_1 + \ldots + \mathbf c_\ell = \mathbf d$, 
let
\[
\binom{\mathbf d}{\mathbf c}:=\frac{\mathbf d!}{\mathbf c!(\mathbf d-\mathbf c)!}=\prod_{a\in Q_0}\binom{d_a}{c_a}\in\mathbb N,
\qquad
\binom{\mathbf d}{\mathbf c_1, \ldots,\mathbf c_\ell}=\frac{\mathbf d!}{\mathbf c_1!\cdots\mathbf c_\ell!}\in\mathbb N.\]
Define the $L_1$-norm on $\mathbb N^{Q_0}$ by $|\mathbf c|:=\sum_{a\in Q_0} c_a\in\mathbb N$.

For each $n\in\mathbb N$, let $\mathbf n:=(n)_{a\in Q_0}\in\mathbb N^{Q_0}$. As established in Section~\ref{section_intro}, the \defn{cardinality vector} of a subset $U\subseteq\overline V$ is $\mathbf U  := (|U_a|)_{a \in Q_0} \in \mathbb{N}^{Q_0}$. Recall that $\mathbf{\overline V} = (|V_a|)_{a\in Q_0}$.
Let $[\mathbf 0,\mathbf{\overline V}]\subseteq\mathbb N^{Q_0}$ be the partially ordered interval between $\mathbf 0$ and $\mathbf{\overline V}$.
 
Next, we consider vectors in $R_{\overline V}^{Q_0}$. For each subset $S\subseteq V_a$, denote the elementary symmetric polynomial of $S$ of degree $1$ as $e_1(S):=\sum_{u\in S}x_u\in R_{\overline V}$.
Extending to subsets $U\subseteq\overline V$, define the polynomial-valued vector $\mathbf e_1(U)$ by
\[\mathbf e_1(U):=(e_1(U_a))_{a\in Q_0}\in R_{\overline V}^{Q_0}.\]
For any polynomial vector $\mathbf q\in R_{\overline V}^{Q_0}$ and any $\mathbf c\in\mathbb N^{Q_0}$, define exponentiation by $\mathbf c$ as
\[
\mathbf q^\mathbf c := 
\prod_{a\in Q_0} q_a^{c_a}  \in  R_{\overline V}.
\]
For example, for any subset $U\subseteq \overline V$, we have
\[\mathbf e_1(U)^\mathbf 1=\prod_{a\in Q_0}\mathbf e_1(U)_a=\prod_{a\in Q_0}\sum_{u\in U_a}x_u.\]

Define the strictly upper triangular \defn{attaching cardinality matrix} $N$. It is indexed on $[\mathbf 0,\mathbf{\overline V}]^2$ with entries in $\mathbb Q$ given by
\begin{equation}
   N_{\mathbf c,\mathbf d}=
    \begin{cases}
        \frac{(-1)^{|\mathbf d-\mathbf c|-1}}{(\mathbf d-\mathbf c)!}
        \mathbf c^{(\mathsf A(Q)(\mathbf d-\mathbf c))}
        &\text{if }\mathbf c<\mathbf d,\\
        \hspace{16mm} 0 & \text{otherwise.}
    \end{cases}
\end{equation}

\begin{example}
    Let $Q$ be the $3$-cycle quiver given in Example~\ref{ex:quiver_rep_3cycle}. The cardinality vector of $V=(V_1,V_2,V_3)$ is given by $\mathbf{\overline V} = (3, 3, 3) \in \mathbb{N}^3$.
    
    The attaching cardinality matrix $N$ is indexed on the interval $[\mathbf{0}, \mathbf{\overline V}]$, which contains exactly $(3+1)^3 = 64$ vectors. To illustrate the multinomial notation, we have $\mathbf{\overline V}! = 3! \cdot 3! \cdot 3! = 216$. Also, $\mathbf e_1(\coprod_{i=1}^3V_i)^\mathbf 2=(x_1+x_2+x_3)^2(y_1+y_2+y_3)^2(z_1+z_2+z_3)^2$.
    
    For any vectors $\mathbf{c} = (c_1, c_2, c_3)$ and $\mathbf{d} = (d_1, d_2, d_3)$ such that $\mathbf{c} < \mathbf{d} \leq \mathbf{\overline V}$, we compute the exponent vector:
    \[\mathsf{A}(Q) (\mathbf{d} - \mathbf{c}) =
    \begin{pmatrix} 0 & 0 & 1 \\ 1 & 0 & 0 \\ 0 & 1 & 0 
    \end{pmatrix}(\mathbf d-\mathbf c)= \begin{pmatrix} d_3 - c_3 \\ d_1 - c_1 \\ d_2 - c_2 \end{pmatrix}.\]
    Therefore, the vector exponentiation evaluates to 
    \[\mathbf{c}^{(\mathsf{A}(Q)(\mathbf{d}-\mathbf{c}))} = c_1^{d_3-c_3} c_2^{d_1-c_1} c_3^{d_2-c_2}.\]
    The non-zero entries of the matrix $N$ are thus
    \begin{equation}
        N_{\mathbf{c}, \mathbf{d}} = \frac{(-1)^{(d_1-c_1) + (d_2-c_2) + (d_3-c_3) - 1}}{(d_1-c_1)! (d_2-c_2)! (d_3-c_3)!} c_1^{d_3-c_3} c_2^{d_1-c_1} c_3^{d_2-c_2}.
    \end{equation}
\end{example}

\section{Weighted enumeration of directed acyclic graphs}
\label{sec:weighted_enumeration_source_factorizable_DAGs}

Recall the constructions in Sections~\ref{subsection_intro_graphs} and~\ref{subsection_stability_source_factor}. Let $\mathcal D$ denote a fixed source-factorizable DAG assignment on a finite graph $G$. We derive a weighted enumeration formula for DAGs in $\mathcal D$. The formula follows from a recursive expression for DAG weights in terms of attaching weights.
To establish this recursion, we make two reductions. In Theorem~\ref{thm_source_factorizable_matrix_enumerator}, we decompose DAG weights into sums over the weights $w(\mathcal B_S(T,U))$, and in Lemma~\ref{Expressing w(DUS) by w(BSTU)-Prop}, we factor the weights ${w}(\mathcal B_S(T,U))$ by attaching weights.

\begin{lemma}
\label{Expressing w(DUS) by w(BSTU)-Prop}
   Let $\mathcal D$ be a source-factorizable DAG assignment on a graph $G$. For any subsets $S,\,T,\,U\subseteq G_V$, the weight of $\mathcal B_S(T,U)$ factors as
  \begin{equation}\label{Weighing BR(T,U) uniformly}
    \begin{aligned}
         {w}(\mathcal B_S(T,U))= {A}({T},U)\cdot  {w}(\mathcal D_S({U\setminus T})).
    \end{aligned}
    \end{equation}
\end{lemma}
\begin{proof}
We construct a weight preserving bijection
\begin{equation}\label{Uniform attaching bijection}
    \mathcal A(T,U)\times\mathcal D_S({U\setminus T})\stackrel{\simeq}{\longrightarrow}
    \mathcal B_S(T,U),
\end{equation}
given by the map $(H',H)\mapsto H'\star H$, with inverse $B\mapsto(B^\out_T, B|_{U \setminus T})$.

Fix any pair $(H',H)\in\mathcal A(T,U)\times\mathcal D_S({U\setminus T})$. Then $H'\star H\in\mathcal B(T,U)$ by definition of $\mathcal A(T,U)$. The sinks in $H'\star H$ are the sinks in $H$, since $H'$ contributes outgoing edges for all vertices in $T$ and no outgoing edges for all vertices in $U\setminus T$. Thus $\Snk(H'\star H)=\Snk(H)=S$, so $H'\star H\in\mathcal B_S(T,U)$. Also
\[(H',H)=((H'\star H)^\out_T,(H'\star H)|_{U\setminus T}),\]
since $H$ contributes no outgoing edges in $T$, and all edges of $H'$ start in $T$.

Fix any $H \in \mathcal{B}_S(T,U)$. Since $T \subseteq \Nss(H)$, all edges of $H$ terminate in $U\setminus T$. So ${U \setminus T}$ is closed under successors in $H$, $H = H^\out_T \star H|_{U \setminus T}$, and $H$, $H|_{U \setminus T}$ have the same outgoing edges on $U\setminus T$. Therefore, $H$ and $H|_{U \setminus T}$ share sinks on $U\setminus T$, so $\Snk(H|_{U \setminus T})=\Snk(H)=S$. By source-factorizability, $H|_{U \setminus T} \in \mathcal{D}(U \setminus T)$, so $H|_{U \setminus T} \in \mathcal{D}_S(U \setminus T)$. Since edges of $H^\out_T$ terminate in $U\setminus T$ and $H^\out_T \star H|_{U \setminus T}=H\in\mathcal B(T,U)$, we have $H^\out_T\in \mathcal{A}(T,U)$.

The graphs in each pair $(H',H)\in \mathcal A(T,U)\times\mathcal D_S({U\setminus T})$ have disjoint edges, so the isomorphism in~\eqref{Uniform attaching bijection} gives
    \begin{equation*}
         {w}(\mathcal B_S(T,U))=\sum_{H'\in\mathcal A(T,U)}\sum_{H\in\mathcal D_S({U\setminus T})} {w}(H') {w}(H)= {A}({T},U)\cdot  {w}(\mathcal D_S({U\setminus T})).
    \end{equation*}
    
    Observe that the proof still holds even if any of the sets $\mathcal B_S(T,U)$, $\mathcal A(T,U)$, or $\mathcal D_S(U\setminus T)$ is empty, which occurs
in particular when $S\not\subseteq U\setminus T$.
\end{proof}

Recall the attaching weight matrix $ {M}$ given by \eqref{eq_W}, and let $I$ denote the identity matrix. Since $ {M}$ is strictly upper triangular, $ {M}$ is nilpotent and $(I- {M})^{-1}$ is invertible and expands as a finite geometric sum in $ {M}$.

\begin{theorem}\label{thm_source_factorizable_matrix_enumerator}
    Let $\mathcal D$ be source-factorizable on a graph $G$. For any vertex subsets $S,\,U\subseteq G_V$, we have
\begin{equation}\label{U,R DAG weight}
         {w}(\mathcal D_S(U)) = (I- {M})^{-1}_{S,U}.
    \end{equation}
\end{theorem}
    \begin{proof}
 Fix $S$ and construct two $\mathcal P(G_V)$-indexed row-vectors in $R_{G_E}$, a vector $\alpha$ to weigh each set $\mathcal D_S(U)$, and a vector $\beta$ to encode the boundary condition. For each $U\subseteq G_V$, define
    \[\alpha_U:= {w}(\mathcal D_S(U)),\quad \beta_U:=
    \begin{cases}
        1&\text{if } U=S,\\
        0&\text{otherwise}.
    \end{cases}\]

    For all subsets $U\subseteq G_V$, we prove the following assertion,
    \begin{equation}\label{Summing by leaf-inclusion}
    \alpha_U=\beta_{U}+\sum_{\emptyset\neq T\subseteq U}(-1)^{|T|-1} {w}(\mathcal B_S(T,U)).
    \end{equation}
    In the first case, assume $S$ is not a proper subset of $U$. Then $\mathcal B_S(T,U)$ is empty for each nonempty subset $T\subseteq U$ since $S\not\subseteq U\setminus T$, so \eqref{Summing by leaf-inclusion} reduces to the assertion $\alpha_U=\beta_{U}$. If $S=U$, then $\mathcal D_S(U)$ contains only the edgeless graph, which has weight $1$. So, $\alpha_U=1=\beta_U$. If $S$ is not contained in $U$, then $\mathcal D_S(U)$ is empty, thus $\alpha_U=0=\beta_U$.
    
    In the remaining case, assume $S\subsetneq U$. Then all DAGs $H\in\mathcal D_S(U)$ have non-sink sources; since $H$ is finite, acyclic, and has edges, it admits nontrivial walk $\omega$ of maximal length with starting vertex $v_1$. Since $\omega$ cannot be extended further, $v_1$ must be a source, and since $\omega$ is nontrivial, $v_1$ is not a sink, so $\Nss(H)$ is nonempty. Consequently, the following binomial expansion is $0$,
    \begin{equation*}
            \sum_{T\subseteq\Nss(H)}(-1)^{|T|}=\sum_{i=0}^{|\Nss(H)|}\binom{|\Nss(H)|}{i}(-1)^i=(1-1)^{|\Nss(H)|}=0,
    \end{equation*}
    which gives the identity
    \begin{equation}\label{Correction terms sum to 1}
            \sum_{\emptyset\neq T\subseteq \Nss(H)}(-1)^{|T|-1}=1.
    \end{equation}
    Thus, we can multiply each summand $ {w}(H)$ of $ {w}(\mathcal D_S(U))$ by~\eqref{Correction terms sum to 1}. We can then swap the order of summation to isolate the weights $ {w}(\mathcal B_S(T,U))$ and recover~\eqref{Summing by leaf-inclusion},
\[
\begin{aligned}
\alpha_U&=\beta_U+\sum_{H\in\mathcal D_S(U)}\left( {w}(H)\sum_{\emptyset\neq T\subseteq \Nss(H)}(-1)^{|T|-1}\right)\\
&=\beta_U+\sum_{\emptyset\neq T\subseteq U}\sum_{H:T\subseteq\Nss(H)}(-1)^{|T|-1} {w}(H)\\
&=\beta_U+\sum_{\emptyset\neq T\subseteq U}(-1)^{|T|-1} {w}(\mathcal B_S(T,U)).
\end{aligned}
\]

Now for each $U\subseteq G_V$, apply Lemma~\ref{Expressing w(DUS) by w(BSTU)-Prop} to~\eqref{Summing by leaf-inclusion} to factor each weight $ {w}(\mathcal B_S(T,U))$
\begin{equation}\label{factored_leaf_sum}
   \alpha_U=\beta_U+\sum_{\emptyset\neq T\subseteq U}\alpha_{U\setminus T}\left((-1)^{|T|-1}  {A}(T,U)\right).
\end{equation}
The summands in~\eqref{factored_leaf_sum} are indexed over the nonempty subsets $T\subseteq U$. We rearrange these terms into a sum over the proper subsets of $U$ by complementation,
\[\{T\subseteq U:T\neq\emptyset\}\to\{T\subsetneq U\},\quad T\mapsto U\setminus T.\]
Reindexing takes~\eqref{factored_leaf_sum} to a sum involving the weighted matrix $ {M}$,
\begin{align*}
\alpha_U&=\beta_U+\sum_{\emptyset\neq T\subseteq U}\alpha_{U\setminus T}\left((-1)^{|T|-1}  {A}(T,U)\right)\\
&=\beta_U+\sum_{T\subsetneq U}\alpha_T\left((-1)^{|U\setminus T|-1} {A}({U\setminus T},U)\right)\\
&=\beta_U+\sum_{T\subsetneq U}\alpha_T {M}_{T,U}\\
&=(\beta+\alpha  {M})_U.
\end{align*}
Since this holds for all $U\subseteq G_V$, we have $\alpha=\beta+\alpha  {M}$. This implies $\alpha=\beta(I- {M})^{-1}$, and thus the result,
    \[ {w}(\mathcal D_S(U))=(\beta(I- {M})^{-1})_U=(I- {M})^{-1}_{S,U}.\qedhere\]
\end{proof}

Since $M$ is strictly upper triangular and $(I-M)^{-1} = I + (I-M)^{-1}M$, we have the following recursive formula for all subsets $S\subseteq U\subseteq G_V$,
\begin{equation}\label{general_recursion_formula}
    w(\mathcal{D}_S(U)) = \delta_{S,U} + \sum_{S \subseteq T \subsetneq U} w(\mathcal{D}_S(T)) M_{T,U}.
\end{equation}
Here, $\delta_{S,U}$ is $1$ if $U=S$, and $\delta_{S,U}$ is $0$ otherwise.

\begin{remark}
\label{subsec:matrix_enumerators_chains}
Let $P$ be a finite partially ordered set. Let $X$ be a strictly upper triangular $P\times P$ matrix with entries in a ring. We recall the notation for chains in a poset from Section~\ref{subsec_partially_ordered_chains}.

As in~\cite[Theorem 4.7.1]{St12}, we make an elementary observation. For all integers $d\in\mathbb N$ and elements $b,\,c\in P$, each entry $(X^d)_{b,c}$ yields a sum over the $(d+1)$-tuples $\Gamma:=(b,\Gamma_1,\ldots,\Gamma_{d-1},c)$ in $P^{d+1}$ that start with $b$ and end with $c$,
\begin{equation}\label{matrix-power-evaluation}
    (X^d)_{b,c}=\sum_{\Gamma}X_{b,\Gamma_1}\cdots X_{\Gamma_{d-1},c}.
\end{equation}
By strict upper triangularity, summands over non-increasing tuples in~\eqref{matrix-power-evaluation} vanish. This leaves a sum over only the tuples $\Gamma$ that correspond to strictly increasing chains from $b$ to $c$ of length $d$.
Consequently, the geometric expansion of the matrix $(I-X)^{-1}$ yields the finite sum
\begin{equation}\label{matrix-chain-equivalence}
    (I-X)^{-1}_{b,c}=\sum_{\Gamma\in\mathcal C(b,c)}X_{b,\Gamma_1}\cdots X_{\Gamma_{L(\Gamma)-1},c}.
\end{equation}
We also can compress the matrix slightly. If $X[b,c]$ is the submatrix of $X$ indexed over the elements $y,\,z$ with $b\leq y,\,z\leq c$, then~\eqref{matrix-chain-equivalence} reduces to the evaluation
\begin{equation}\label{reduced-submatrix-chain-equivalence}
    (I-X[b,c])^{-1}_{b,c}=\sum_{\Gamma\in\mathcal C(b,c)}X_{b,\Gamma_1}\cdots X_{\Gamma_{L(\Gamma)-1},c}=(I-X)^{-1}_{b,c}.
\end{equation}
\end{remark}

The power set $\mathcal P(G_V)$ is finite and partially ordered, and the attaching weight matrix $M$ is strictly upper triangular, so we can apply~\eqref{matrix-chain-equivalence} to expand $(I- M)^{-1}$ by strictly increasing chains in the partially ordered power set $\mathcal P(G_V)$. By Theorem~\ref{thm_source_factorizable_matrix_enumerator}, we have the following.

\begin{corollary}\label{Hyperforest enumerator chain formulation-Prop}
    If $\mathcal D$ is source-factorizable on $G$, then for all subsets $S\subseteq U \subseteq G_V$, we have the following sum over chains in $\mathcal P(G_V)$,
    \begin{equation*}
    \begin{split}
         {w}(\mathcal D_S(U)) &= \sum_{\Gamma\in\mathcal{C}(S,U)}\prod_{\ell=1}^{L(\Gamma)} \left((-1)^{|\Gamma_\ell\setminus \Gamma_{\ell-1}|-1} {A}({\Gamma_\ell\setminus \Gamma_{\ell-1}},{\Gamma_{\ell}})\right)\\
        &=\sum_{\Gamma\in\mathcal{C}(S,U)}(-1)^{|U\setminus S|-L(\Gamma)}\prod_{\ell=1}^{L(\Gamma)}  {A}({\Gamma_\ell\setminus \Gamma_{\ell-1}},{\Gamma_{\ell}}).
    \end{split}
    \end{equation*}
\end{corollary}

\section{Eventually constant set-valued representations}
\label{sec:eventually_constant_set-valued_representations}

Fix a finite quiver $Q$, whose vertex set $Q_0$ indexes a tuple $(V_a)_{a \in Q_0}$ of finite, nonempty, disjoint sets, and let $\overline V$ be the disjoint union of $V_a$, where $a\in Q_0$. Assume $Q$ has no sinks and each vertex $a\in Q_0$ is the target of a path of length $\geq|Q_0|$.

\subsection{Representable graphs}
\label{subsec:applying_source_factorizability_to_quiver_reps}

In this section, we construct a DAG assignment and relate it to the set of stabilizing representations. Then, we prove the assignment is source-factorizable and calculate its attaching weights. We use the constructions on quivers given in Sections~\ref{subsection_stability_source_factor} and~\ref{subsection_quivers}, and the graph operations given in Section~\ref{subsection_intro_graphs}.

For each vertex subset $U\subseteq \overline V$, let $\DAG(U)$ be the set of directed acyclic subgraphs of $G(Q)$ on $U$. Recall $G(Q)$ is the graph on $\overline V$ with $Q_1$-labeled edges, and for each $U\subseteq\overline V$, $\mathbf U:=(|U_a|)_{a\in Q_0}\in\mathbb N^{Q_0}$. Observe $\mathbf U\geq\mathbf 1$ if and only if each set $U_a$ is nonempty, or equivalently if $U$ contains a transversal of $V$.

\begin{definition}
     We define a DAG assignment $\mathcal D$ on the graph $G(Q)$ as follows. For each subset $U\subseteq \overline V$, let $\mathcal D(U)\subseteq\DAG(U)$ be the set of representable directed acyclic subgraphs on $U$.
     Recall that a graph $H$ on $U$ is representable if there exists a subset $T\subseteq U$ and a tuple of functions $(W_\alpha:T_{s(\alpha)}\to U_{t(\alpha)})$ such that $H=G(W)$. The graphs induced by such tuples form the set $\mathcal R_U(T\to U)$, and when we restrict to such tuples with each $\im W_\alpha\subseteq S_{t(\alpha)}$, we have the set $\mathcal R_U(T\to S)$. Stated without reference to tuples, $H\in\mathcal R_U(T\to U)$ if for each arrow $\alpha\in  Q_1$ and each vertex $u\in T_{s(\alpha)}$, $H$ has a unique edge $(u,v)^{(\alpha)}$ to some vertex $v\in U_{t(\alpha)}$, and no other edges.
     
     When $\mathbf U\geq \mathbf 1$, then $H$ is representable if and only if there exists a representation $W\in\Rep(Q,(U_a))$ such that $H=G(W)^\out_T$. In this case, $G(W)^\out_T=G((W_\alpha|_{T_{s(\alpha)}})_{\alpha\in Q_1})$.
\end{definition}

Now, we isolate two graph-theoretic observations. First, let $H$ be a subgraph of $G(Q)$ and suppose $T\subseteq H_V$ is closed under successors. Then $H^\out_{H_V\setminus T}$ is acyclic if and only if all sufficiently long walks in $H$ have target in $T$. If $H^\out_{H_V\setminus T}$ is acyclic, then all sufficiently long walks in $H$ cannot be entirely contained in $H^\out_{H_V\setminus T}$. Such walks must eventually reach a vertex in $T$, and thus their target also lies in $T$ by closure. Conversely, if $H^\out_{H_V\setminus T}$ is not acyclic, then it contains a directed cycle. Since all edges in $H^\out_{H_V\setminus T}$ have their source in $H_V\setminus T$, this cycle cannot contain any vertex in $T$. Thus, $H$ admits arbitrarily long walks whose targets never enter $T$.

Next, consider any subsets $T\subseteq U\subseteq \overline V$ and $G(W)\in\mathcal R_U(T\to U)$. Since $Q$ has no sinks, $\Snk(G(W))=U\setminus T$; for each vertex $u\in T_a$, there exists an arrow $\alpha:a\to b$ in $Q$, so $G(W)$ has the outgoing edge $(u, W_\alpha(u))^{(\alpha)}$. Therefore, $\mathcal D_{U\setminus T}(U)=\mathcal D(U)\cap \mathcal R_U(T\to U)$.

By the following Lemma, representations that stabilize in a fixed subset $S\subseteq U$ decompose via a weight preserving bijection into a DAG in $\mathcal D _S(U)$ and a representation on $S$.

\begin{lemma}\label{lem:representations_decompose}
Fix subsets $S\subseteq U$ in $\overline V$ with $\mathbf 1\leq \mathbf S$. Then there is a weight preserving bijection between subgraphs of $G(Q)$,
\[
\{ G(W) : W\in\EC_S(U)\}\stackrel{\simeq}{\longrightarrow}\mathcal D _S(U)\times\mathcal R(S),
\]
    given by $G(W)\mapsto (G(W)^\out_{U\setminus S}, G(W)|_S)$, with inverse $(H,H')\mapsto H\star H'$. Thus,
    \begin{equation}
         \sum_{W\in\EC_S(U)}w(G(W))= {w}(\mathcal D _S(U)) {w}(\mathcal R(S)),
    \end{equation}
    which we abbreviate as $w(\EC_S(U))$.
\end{lemma}
\begin{proof}
Fix any representation $W\in\EC_S(U)$. Since $W$ stabilizes in $S$, the set $S$ is closed under successors in $G(W)$, $G(W)=G(W)^\out_{U\setminus S}\star G(W)|_S$, and all sufficiently long walks in $G(W)$ have target in $S$. Thus $G(W)^\out_{U\setminus S}$ is acyclic. Since $Q$ has no sinks, we know $\Snk(G(W)^\out_{U\setminus S})=S$, so $G(W)^\out_{U\setminus S}\in\mathcal D _S(U)$. The maps $(W_\alpha)_{\alpha\in Q_1}$ are closed on $S$, so $W$ induces a subrepresentation in $\Rep(Q,(S_a)_{a\in Q_0})$, $W':=(W_\alpha':S_{s(\alpha)}\to S_{t(\alpha)})_{\alpha\in Q_1}$.
We then see $G(W)|_S=G(W')$, so $G(W)|_S\in\mathcal R(S)$.

Conversely, fix any pair $(G(W),G(W'))\in \mathcal D _S(U)\times\mathcal R(S)$, with $W'\in\Rep(Q,(S_a)_{a\in Q_0})$ and $W:=(W_\alpha:(U\setminus S)_a\to U_a)_{\alpha\in Q_1}$. For each arrow $\alpha\in Q_1$, the maps $W_\alpha$ and $W'_\alpha$ glue to form a map $X_\alpha:U_{s(\alpha)}\to U_{t(\alpha)}$, by
\[X_\alpha(u):=
\begin{cases}
    W_\alpha(u)&\text{if }u\in U_{s(\alpha)}\setminus S_{s(\alpha)},\\
    W'_\alpha(u)&\text{if }u\in S_{s(\alpha)}.
\end{cases}\]
Then, $X:=(X_\alpha)_{\alpha\in Q_1}$ is in $\Rep(Q,(U_a)_{a\in Q_0})$ and
$G(W)\star G(W')=G(X)$.
For each arrow $\alpha\in Q_1$, $X_\alpha$ extends $W'_\alpha$ on $S_{s(\alpha)}\subseteq U_{s(\alpha)}$, so $X_\alpha$ maps $S_{s(\alpha)}$ into $S_{t(\alpha)}$. Also, $S$ is closed on $G(X)$, $G(X)^\out_{U\setminus S}=G(W)$, and $G(W)$ is acyclic, so all sufficiently long walks in $G(X)$ have target in $S$. Thus $X$ stabilizes in $S$, so $X\in\EC_S(U)$.
\end{proof}

To make use of Lemma~\ref{lem:representations_decompose}, we need to weigh the sets $\mathcal R(S)$ and $\mathcal D _S(U)$. We next prove that $\mathcal D$ is source-factorizable so that we can decompose $w(\mathcal D _S(U))$ by its attaching weights.

Fix any subsets $T\subseteq U\subseteq \overline V$. For the nontriviality axiom, the edgeless graph is in the subset $\mathcal R_U(\emptyset\to U)$ of $\mathcal D(U)$.
For the source-removal axiom, we fix $G(W)\in\mathcal B(T,U)$ in $\mathcal R_U(S\to U)$ and need to show $G(W)|_{U\setminus T}\in\mathcal D(U\setminus T)$. First, $G(W)|_{U\setminus T}$ is acyclic since it is a subgraph of $G(W)$. Second, $T\subseteq\Nss(G(W))$, so the vertices in $T$ are sources in $G(W)$, so each map $W_\alpha:S_{s(\alpha)}\to U_{t(\alpha)}$ restricts to a map $W'_\alpha:(S\setminus T)_{s(\alpha)}\to(U\setminus T)_{t(\alpha)}$. This yields a tuple of functions $W':=(W'_\alpha)_{\alpha\in Q_1}$, and $G(W)|_{U\setminus T}=G(W')$, so $G(W)|_{U\setminus T}$ is representable. Thus $G(W)|_{U\setminus T}\in\mathcal D(U\setminus T)$.

The following lemma verifies the remaining uniformity axiom. This proves that $\mathcal D $ is source-factorizable and gives the form of the attaching sets.
\begin{lemma}\label{EC_is_source_factorizable}
For all subsets $T\subseteq U\subseteq \overline V$ and each DAG $H\in\mathcal D(U\setminus T)$, the attaching set $\mathcal A_H(T,U)$ is $\mathcal R_U(T\to U\setminus T)$.
\end{lemma}
\begin{proof}
Suppose $S\subseteq U\setminus T$ is the subset such that $H\in\mathcal R_{U\setminus T}(S\to U\setminus T)$, and fix any $H'\in\mathcal R_U(T\to U\setminus T)$. Since $T$ and $S$ are disjoint, the tuples of maps representing $H$ and $H'$ glue to form a tuple representing $H'\star H$ as in Lemma~\ref{lem:representations_decompose}. Also, all edges in both $H'$ and $H$ have targets in $U\setminus T$, and each vertex in $T$ has an outgoing edge in $H'$, so $T\subseteq\Nss( H' \star H)$. Finally, $H' \star H$ is acyclic since walks are bounded in length; each walk enters $U\setminus T$ in at most one step and then follows a walk of bounded length in the acyclic graph $H$. So, $ H' \star H\in\mathcal B(T,U)$, and thus $ H' \in\mathcal A_H(T,U)$.

Conversely, let $H'\in\mathcal A_H(T,U)$. Then $H'\star H\in\mathcal B(T,U)$, so $H'\star H\in\mathcal R_U(S\to U)$ for some subset $S\subseteq U$. Since no vertices in $T$ are sinks in $H'\star H$, $T$ is disjoint from $\Snk(H'\star H)=U\setminus S$, so $T\subseteq S$. Thus for each arrow $\alpha\in Q_1$ and each vertex $u\in T_{s(\alpha)}$, $H'\star H$ has a single outgoing edge $(u,v)^{(\alpha)}$ to a vertex $v\in U_{t(\alpha)}\setminus T=(U\setminus T)_{t(\alpha)}$. This edge belongs to $H'$ since $u\not\in U\setminus T$, so $H'\in\mathcal R_U(T\to U\setminus T)$.
\end{proof}

Fix any subsets $S,\,T\subseteq U\subseteq\overline V$. It remains to calculate $\mathcal R(S)$ and the set of attaching graphs, $\mathcal A(T,U)$. Since $\mathcal R_S(S\to S)=\mathcal R(S)$ and $\mathcal R_U(T\to U\setminus T)=\mathcal A(T,U)$, we can determine both weights by calculating $w(\mathcal R_U(T\to S))$.

We decompose the graphs in $\mathcal R_U(T\to S)$ by the arrows $\alpha\in Q_1$.
For each arrow $\alpha\in Q_1$, associate each function $f:T_a\to S_b$ with its $\alpha$-indexed graph $f^{(\alpha)}$ on $U$ with the edge set
$\{ (u,f(u))^{(\alpha)} : u\in T_{a}\}$.
Define the set of graphs $\hom(T_a,S_b)^{(\alpha)}:=\{f^{(\alpha)} : f\in\hom(T_a,S_b)\}$ and note that its weight factors as
\begin{equation}\label{weight_edge_labeled_hom}
    {w}(\hom(T_{a},S_{b})^{(\alpha)})=\sum_{f:T_{a}\to S_{b}} \prod_{u\in T_{a}}x_{u,f(u)}^{(\alpha)}=\prod_{u\in T_{a}}\sum_{v\in S_{b}}x_{u,v}^{(\alpha)}.
\end{equation}

There is a weight preserving bijection,
\[\mathcal R_U(T\to S)\stackrel{\simeq}{\longrightarrow}\prod_{\alpha\in Q_1}\hom(T_{s(\alpha)},S_{t(\alpha)})^{(\alpha)},\]
by $G(W)\mapsto (W_{\alpha}^{(\alpha)})_{\alpha\in Q_1}$. Its inverse is given by sending $(f^{(\alpha)})_{\alpha\in Q_1}$ to the edge-disjoint union over all arrows $\alpha\in Q_1$.
Since the edges of each map $f^{(\alpha)}$ are pairwise disjoint, the weight factors as
\begin{equation}\label{R_U(S,T)_weight_edges}
\begin{aligned}
{w}(\mathcal R_U(T\to S))=\prod_{\alpha\in Q_1} {w}(\hom(T_{s(\alpha)},S_{t(\alpha)})^{(\alpha)})=\prod_{\alpha:a\to b}\prod_{u\in T_{a}}\sum_{v\in S_b}x_{u,v}^{(\alpha)}.
\end{aligned}
\end{equation}

Finally, we can weigh $\mathcal R(S)$ and $\mathcal A(T,U)$. Since $\mathcal R(S)=\mathcal R_S(S\to S)$, the weight of the representations on $S\subseteq \overline V$ is given by
\begin{equation}\label{directed_rep_weight}
     {w}(\mathcal R(S))=\prod_{\alpha:a\to b}\prod_{u\in S_{a}}\sum_{v\in S_b}x_{u,v}^{(\alpha)}.
\end{equation}
Since $\mathcal A(T,U)=\mathcal R_U(T\to U\setminus T)$, the attaching weight of $T\subseteq U\subseteq \overline V$ is given by
\begin{equation}\label{directed_attaching_weight}
    {A}(T,U)=w(\mathcal A(T,U))=\prod_{\alpha:a\to b}\prod_{u\in T_{a}}\sum_{v\in (U\setminus T)_{b}}x_{u,v}^{(\alpha)}.
\end{equation}

\subsection{Generating function in edge variables}
\label{subsec:gen_fun_edges}

In this section, we derive the generating function $P(\mathbf x)$ in directed edge variables. We use the notation established in Sections~\ref{subsection_stability_source_factor},~\ref{subsection_quivers}, and~\ref{subsec:multinomial_notation}, and the results from the above Section~\ref{subsec:applying_source_factorizability_to_quiver_reps}.

By~\eqref{directed_attaching_weight}, the entries of the attaching weight matrix $M$ are given by
\begin{equation}\label{eq_W_quiver_full_var}
     {M}_{T,U}=
    \begin{cases}
    (-1)^{|U\setminus T|-1}\prod_{\alpha:a\to b}\prod_{u\in (U\setminus T)_{a}}\sum_{v\in T_{b}}x_{u,v}^{(\alpha)}&\text{if } T\subsetneq U,\\
    \hspace{18mm} 0&\text{otherwise.}
    \end{cases}
\end{equation}

Consider any subsets $S\subseteq U\subseteq \overline V$. Theorem~\ref{thm_source_factorizable_matrix_enumerator} gives the weighted enumerator for $\mathcal D_S(U)$ in terms of the attaching weight matrix $M$,
and~\eqref{directed_rep_weight} gives the weight of $\mathcal R(S)$. So, by Lemma~\ref{lem:representations_decompose}, the weight of $\EC_S(U)$ is
\[ {w}(\EC_S(U))=\left(\prod_{\alpha:a\to b}\prod_{u\in S_{a}}\sum_{v\in S_b}x_{u,v}^{(\alpha)}\right)(I- {M})^{-1}_{S,U}.\]
Expanded by chains in the power set lattice $\mathcal P(\overline V)$, Corollary~\ref{Hyperforest enumerator chain formulation-Prop} gives
   \begin{equation*}
   \begin{aligned}
         {w}(\EC_S(U)) &= w(\mathcal R(S))\sum_{\Gamma\in\mathcal{C}(S,U)}(-1)^{|U\setminus S|-L(\Gamma)}\prod_{\ell=1}^{L(\Gamma)} {A}({\Gamma_\ell\setminus \Gamma_{\ell-1}},{\Gamma_{\ell}})\\
         &= \prod_{\alpha:a\to b}\prod_{u\in S_{a}}\sum_{v\in S_b}x_{u,v}^{(\alpha)}\sum_{\Gamma\in\mathcal{C}(S,U)}(-1)^{|U\setminus S|-L(\Gamma)}\prod_{\ell=1}^{L(\Gamma)} 
        \prod_{\alpha:a\to b}\prod_{u\in ({\Gamma_\ell\setminus \Gamma_{\ell-1}})_{a}}\sum_{v\in (\Gamma_{\ell-1})_{b}}x_{u,v}^{(\alpha)}.
    \end{aligned}
    \end{equation*}

To weigh $\EC$ in terms of the sets $\EC_S(U)$, we note that $\EC$ is partitioned by the sets
\[\{\EC_C(\overline V) : C \in \mathcal{T}(V)\}.\]
Consider any representation $W\in \EC$. For all sufficiently long paths $p$ in $Q$ with target $b$, $\im W_p$ is contained within a singleton of $\{c_b\}\subseteq V_b$. Also, this singleton is unique since sufficiently long paths with target $b$ exist by assumption on $Q$. The collection $(c_b)_{b\in Q_0}$ uniquely determines the transversal $C$ for which $W\in\EC_C(\overline V)$, verifying pairwise disjointness.

Consequently, the generating function expands as the sum over these transversals,
\begin{equation}\label{eq:transversals_partition_EC}
 {P}(\mathbf x)=\sum_{C\in\mathcal T(V)}w(\EC_{C}(\overline V))=\sum_{C\in\mathcal T(V)} \left(\prod_{\alpha:a\to b}x_{c_{a},c_{b}}^{(\alpha)}\right)(I- {M})^{-1}_{C,\overline V}.
\end{equation}

We now establish that the generating function $P(\mathbf{x})$ is invariant under permutations of the vertices within each set $V_a$. Let $\mathfrak{S}_{a}$ denote the symmetric group on the set $V_a$ for each vertex $a \in Q_0$, and let $\mathfrak{S} := \prod_{a \in Q_0} \mathfrak{S}_{a}$ denote the product group. The natural action of $\mathfrak S$ on the vertex set $\overline V$ is given by ${\sigma} \cdot u:= \sigma_a(u)$ for $u \in V_a$. This induces an action on the polynomial ring of edge variables $R_{G(Q)_E}$. For any ${\sigma} = (\sigma_a)_{a \in Q_0} \in \mathfrak{S}$ and any arrow $\alpha: a\to b$ in $Q_1$, this induced action is given by
\begin{equation}
    {\sigma} \cdot x_{u,v}^{(\alpha)} := x_{\sigma_a(u), \sigma_b(v)}^{(\alpha)}.
\end{equation}

In~\eqref{eq:transversals_partition_EC}, the generating function expands as a sum over the weights of the sets of representations stabilizing in each transversal. Fortunately, we do not have to calculate each of these weights independently, since for each transversal $C \in \mathcal{T}(V)$, ${\sigma} \cdot w(\EC_C(\overline V)) = w(\EC_{{\sigma}(C)}(\overline V))$. To see this, note that for all subsets $S, T \subseteq U \subseteq \overline V$
\begin{equation}
    {\sigma} \cdot w(\mathcal{R}_U(T \to S)) = \prod_{\alpha \in Q_1} \prod_{u \in T_{s(\alpha)}} \sum_{v \in S_{t(\alpha)}} x_{\sigma_{s(\alpha)}(u), \sigma_{t(\alpha)}(v)}^{(\alpha)} = w(\mathcal{R}_{{\sigma}(U)}({\sigma}(T) \to {\sigma}(S))).
\end{equation}
Thus, the action commutes with the attaching weights, yielding ${\sigma} \cdot A(U \setminus T, U) = A({\sigma}(U) \setminus {\sigma}(T), {\sigma}(U))$, and ${\sigma} \cdot w(\mathcal{R}(C)) = w(\mathcal{R}({\sigma}(C)))$. By the chain formulation, we have
\begin{align*}
    {\sigma} \cdot w(\EC_C(\overline V))&= w(\mathcal{R}({\sigma}(C))) \sum_{\Gamma \in \mathcal{C}(C, \overline V)} (-1)^{|\overline V \setminus C| - L(\Gamma)} \prod_{\ell=1}^{L(\Gamma)} A({\sigma}(\Gamma_\ell) \setminus {\sigma}(\Gamma_{\ell-1}), {\sigma}(\Gamma_\ell))\\
    &= w(\EC_{{\sigma}(C)}(\overline V)).
\end{align*}

The group $\mathfrak{S}$ acts transitively on the set of transversals $\mathcal{T}(V)$. Also, for each transversal $C$, a permutation ${\sigma} \in \mathfrak{S}$ fixes $C$ if and only if it fixes the single element $c_a$ in each set $V_a$. Thus, the size of the stabilizer of $C$ is $\prod_{a \in Q_0} (|V_a| - 1)! = (\mathbf{\overline V} - \mathbf{1})!$, where $\mathbf{1}=(1)_{a\in Q_0}$. By the orbit-stabilizer theorem, the generating function is an orbit sum over any fixed transversal $C_0 \in \mathcal{T}(V)$,
\begin{equation}\label{orbit_sum_edges}
    P(\mathbf{x}) = \frac{1}{(\mathbf{\overline V} - \mathbf{1})!} \sum_{{\sigma} \in \mathfrak{S}} {\sigma} \cdot w(\EC_{C_0}(\overline V)).
\end{equation}
Equivalently, we can write the generating function as a sum over the left cosets of the stabilizer,
\begin{equation}
    P(\mathbf{x}) = \sum_{\overline{{\sigma}} \in \mathfrak{S}/\text{Stab}(C_0)} {\sigma} \cdot w(\EC_{C_0}(\overline V)).
\end{equation}
The generating function is invariant under this action, since for any ${\mu} \in \mathfrak{S}$,
\begin{equation}
    {\mu} \cdot P(\mathbf{x}) = \frac{1}{(\mathbf{\overline V} - \mathbf{1})!} \sum_{{\sigma} \in \mathfrak{S}} {\mu}  {\sigma} \cdot w(\EC_{C_0}(\overline V)) = P(\mathbf{x}).
\end{equation}

\begin{remark}
    First, consider a transversal $C\in\mathcal T(V)$ and a subset $C\subseteq U\subseteq \overline V$. By Lemma~\ref{lem:representations_decompose}, $w(\EC_C(U)) = w(\mathcal{D}_C(U)) w(\mathcal{R}(C))$. Multiplying the recurrence~\eqref{general_recursion_formula} for the weight $w(\mathcal D_C(U))$ by the sink weight $w(\mathcal{R}(C))$ yields
    \begin{equation*}
        w(\EC_C(U)) = \delta_{C,U} w(\mathcal{R}(C)) + \sum_{C \subseteq T \subsetneq U} w(\EC_C(T)) M_{T,U}.
    \end{equation*}

    Recall that $P_U(\mathbf x)$ is the generating function over all eventually constant representations on $U$. Since $\mathbf U  \geq \mathbf{1}$, the transversals of $U$ are the transversals of $V$ contained in $U$, so $P_U(\mathbf{x}) = \sum_{C \in \mathcal{T}(V), C \subseteq U} w(\EC_C(U))$. Thus, we have the sum over all transversals $C\in\mathcal T(V)$ contained in $U$,
   \[P_U(\mathbf x) = \sum_{C \subseteq U} \delta_{C,U} w(\mathcal{R}(C)) + \sum_{C \subseteq U} \sum_{C \subseteq T \subsetneq U} w(\EC_C(T)) M_{T,U}.\]
   Swapping the order of summation gives
        \begin{align*}
        P_U(\mathbf x)&= \sum_{C \subseteq U} \delta_{C,U} w(\mathcal{R}(C)) + \sum_{\substack{T \subsetneq U\\\mathbf T \geq\mathbf 1}} \left( \sum_{C \subseteq T} w(\EC_C(T)) \right) M_{T,U} \\
        &= \sum_{C \subseteq U} \delta_{C,U} w(\mathcal{R}(C)) + \sum_{\substack{T \subsetneq U\\\mathbf T \geq\mathbf 1}} P_T(\mathbf{x}) M_{T,U}.
    \end{align*}
When $U\in\mathcal T(V)$, then $P_U(\mathbf x)=w(\mathcal R(U))$. When $\mathbf U >\mathbf 1$, then the $\delta_{C,U}$ terms vanish. Expanding the $M_{T,U}$ entries gives
\begin{equation}\label{eq:edge_recurrence}
    P_U(\mathbf x)=\sum_{\substack{T \subsetneq U\\\mathbf T \geq\mathbf 1}} (-1)^{|U\setminus T|-1} P_T(\mathbf{x}) \prod_{\alpha:a\to b}\prod_{u\in (U\setminus T)_{a}}\sum_{v\in T_{b}}x_{u,v}^{(\alpha)}.
\end{equation}
\end{remark}

\begin{remark}
    Consider subsets $X,Y$ that partition the vertex set $Q_0$ and are closed under successors. In this remark, we add a superscript to indicate the quiver with respect to which each construction is being carried out.
    Let $U:=\coprod_{a\in X}V_a$ and $T:=\coprod_{a\in Y}V_a\subseteq\overline V$. Observe that $U$ and $T$ also partition $\overline V$ and are closed under successors in the graph $G(Q)$.
    
Since $X$ and $Y$ are closed under successors in $Q$, the subsets induce subquivers of $Q$, say $Q(X)$ and $Q(Y)$. Also, there is a bijective correspondence
\[
\Rep(Q,V) \stackrel{\simeq}{\longrightarrow} \Rep(Q(X),(V_a)_{a\in X})\times\Rep(Q(Y),(V_a)_{a\in Y})
\]
given by $W\mapsto(W|_{Q(X)},W|_{Q(Y)})$. The map induces a weight-preserving bijection
\[
\mathcal R(\overline V)\cong\mathcal R^{Q(X)}(U)\times \mathcal R^{Q(Y)}(T).
\]

Consider any representation $W\in\Rep(Q,V)$. If $W$ is eventually constant, then $W$ stabilizes in a transversal $C\in\mathcal T(V)$. So for all sufficiently long paths $p$ in $Q(X)$, $W_p$ maps into $C$, so $(W|_{Q(X)})_p$ maps into $C\cap U \in \mathcal{T}((V_a)_{a\in X})$. Therefore $W|_{Q(X)}$ is eventually constant, and $W|_{Q(Y)}$ is similarly. Conversely, if $W|_{Q(X)}$ and $W|_{Q(Y)}$ are eventually constant on $C\in \mathcal{T}((V_a)_{a\in X})$ and $C' \in \mathcal{T}((V_a)_{a\in Y})$, then for all sufficiently long paths $p$ in $Q$, either $p$ lies entirely in $X$ or $p$ lies entirely in $Y$. Thus, either $(W|_{Q(X)})_p$ maps into $C$ or $(W|_{Q(Y)})_p$ maps into $C'$, so $W_p$ maps into $C\cup C'\in\mathcal T(V)$, so $W$ is eventually constant.

 By this correspondence, the generating function factors $P^Q_{\overline V}(\mathbf{x}) = P^{Q(X)}_U(\mathbf{x}) P^{Q(Y)}_T(\mathbf{x})$. Consequently, if $X_1, \ldots,X_k$ partition the vertex set $Q_0$ and are successor closed with induced subquivers $Q(1),\ldots,Q(k)$ with $U(1):=\coprod_{a\in X_1}V_a,\ldots,U(k):=\coprod_{a\in X_k}V_a$, then
    \begin{equation}
        P_{\overline V}^Q(\mathbf{x}) = \prod_{i=1}^k P^{Q(i)}_{U(i)}(\mathbf{x}).
    \end{equation}
\end{remark}

\subsection{Generating function in vertex variables}
\label{subsec:gen_function_vertex}
We reinterpret the results in Section~\ref{subsec:gen_fun_edges} to target variables via the specialization $x_{u,v}^{(\alpha)}\mapsto x_v$. We use the notation given in Sections~\ref{subsection_stability_source_factor},~\ref{subsection_quivers}, and~\ref{subsec:multinomial_notation}, and the results from the above Section~\ref{subsec:gen_fun_edges}.

We first consider any subsets $S,\,T\subseteq U\subseteq \overline{V}$ and the weight $w(\mathcal R_U(T\to S))$. Under the target specialization, the weight in~\eqref{R_U(S,T)_weight_edges} evaluates to
\begin{equation}\label{w(R_U(TtoS))_var_unsimplified}
    w(\mathcal R_U(T\to S))=\prod_{\alpha:a\to b} e_1(S_{b})^{|T_{a}|}\in R_{\overline{V}}.
\end{equation}
We can further simplify~\eqref{w(R_U(TtoS))_var_unsimplified} by reindexing with vertices instead of arrows. For each pair of vertices $a,\,b\in Q_0$, the factor $e_1(S_{b})^{|T_a|}$ appears once for every arrow $\alpha:a\to b$ in $Q_1$. This number is given by the entry $\mathsf A(Q)_{b,a}$ of the adjacency matrix of $Q$. Multiplying these factors as $a,\,b\in Q_0$ range over the vertices yields
\begin{equation}\label{R_U(S,T)_weight_targets_2}
 w(\mathcal R_U(T\to S))= \prod_{b\in Q_0} e_1(S_b)^{\sum_{a\in Q_0} \mathsf A(Q)_{b,a}|T_a|} = \prod_{b\in Q_0}\mathbf  e_1(S)_b^{(\mathsf A(Q) \mathbf T )_b} =\mathbf e_1(S)^{\mathsf A(Q)\mathbf T }.
\end{equation}

Therefore, for all subsets $T,\,U\subseteq \overline{V}$, the specialized transition weight matrix has the entry
\begin{equation}\label{eq_W_quiver_tar_var}
	M_{T,U}=
	\begin{cases}
	(-1)^{|U\setminus T|-1}\mathbf e_1(T)^{(\mathsf A(Q)(\mathbf U -\mathbf T ))}
    &\text{if } T\subsetneq U,\\
	\hspace{18mm} 0&\text{otherwise.}
	\end{cases}
\end{equation}
Also, for any subset $S\subseteq \overline{V}$, the weight of $\mathcal R(S)$ is $\mathbf e_1(S)^{\mathsf A(Q)\mathbf S }$. So, the weight of $\EC_S(U)$ specializes to $\mathbf e_1(S)^{\mathsf A(Q)\mathbf S }(I-M)^{-1}_{S,U}$. Over chains of subsets, this is
\begin{equation}\label{eq:target_specialized_DSU_chain}
    w(\EC_S(U)) =\mathbf e_1(S)^{\mathsf A(Q)\mathbf S } \sum_{\Gamma\in\mathcal{C}(S,U)}(-1)^{|U\setminus S|-L(\Gamma)}\prod_{\ell=1}^{L(\Gamma)} \mathbf e_1(\Gamma_{\ell-1})^{\mathsf A(Q)(\bm{\Gamma}_\ell - \bm{\Gamma}_{\ell-1})}.
\end{equation}

For each transversal $C\in\mathcal T(V)$ and each vertex $b\in Q_0$, we have $(\mathsf A(Q)\mathbf 1)_b=d^{Q}_{\In}(b)$ and $e_1(C_b)=x_{c_b}$. Thus, in terms of the specialized matrix $M$ and in terms of chains, the generating function is
\begin{equation}\label{eq:target_specialized_gen_fun_chains}
\begin{aligned}
    P(\mathbf x)&=\sum_{C\in\mathcal T(V)}\prod_{b\in Q_0}x_{c_b}^{d^Q_{\In}(b)}(I-M)^{-1}_{C,\overline{V}}\\
    &=\sum_{C\in\mathcal T(V)}\prod_{b\in Q_0}x_{c_b}^{d^Q_{\In}(b)} \sum_{\Gamma\in\mathcal{C}(C,\overline{V})}(-1)^{|\overline{V}\setminus C|-L(\Gamma)}\prod_{\ell=1}^{L(\Gamma)} \mathbf e_1(\Gamma_{\ell-1})^{\mathsf A(Q)(\bm{\Gamma}_\ell - \bm{\Gamma}_{\ell-1})}.
    \end{aligned}
\end{equation}

We now establish that the specialized generating function is invariant under permutations of the variables within each set $V_a$. Recall the group $\mathfrak{S} = \prod_{a \in Q_0} \mathfrak{S}_{a}$ and its action on $R_{G(Q)_E}$ from Section~\ref{subsec:gen_fun_edges}. This group induces a natural action on the polynomial ring $R_{\overline{V}}$, given by ${\sigma} \cdot x_v := x_{\sigma_a(v)}$ for $v \in V_a$. 

Let $\varphi: R_{G(Q)_E} \to R_{\overline{V}}$ be the target specialization homomorphism, defined by $\varphi(x_{u,v}^{(\alpha)}) = x_v$. This homomorphism is $\mathfrak{S}$-equivariant since for all ${\sigma} \in \mathfrak{S}$, arrows $\alpha:a \to b$, and vertices $u \in V_a$, $v \in V_b$, we see
\begin{equation*}
    \varphi({\sigma} \cdot x_{u,v}^{(\alpha)}) = \varphi(x_{\sigma_a(u), \sigma_b(v)}^{(\alpha)}) = x_{\sigma_b(v)} = {\sigma} \cdot x_v = {\sigma} \cdot \varphi(x_{u,v}^{(\alpha)}).
\end{equation*}
Consequently, $\varphi$ maps the $\mathfrak{S}$-invariant subring of $R_{G(Q)_E}$ into the $\mathfrak{S}$-invariant subring of $R_{\overline{V}}$. Since the edge-variable generating function is invariant, its image $P(\mathbf{x}) \in R_{\overline{V}}$ remains invariant under permutations of the variables $\{x_v : v \in V_a\}$ for each $a \in Q_0$. Also, the orbit sum~\eqref{orbit_sum_edges} passes to the same formula under this specialization,
\begin{equation}
    P(\mathbf{x}) = \frac{1}{(\mathbf{\overline V} - \mathbf{1})!} \sum_{{\sigma} \in \mathfrak{S}} {\sigma} \cdot w(\EC_{C_0}(\overline{V})).
\end{equation}

\begin{remark}\label{rmk:target_recursion}
    We revisit the recurrence~\eqref{eq:edge_recurrence} in the target variable specialization. If $U$ is a transversal, then
    \[P_U(\mathbf{x}) = \prod_{b\in Q_0}x_{u_b}^{d^{Q}_{\In}(b)}.\]
    If $U$ strictly contains a transversal, then we have
\begin{equation}\label{eq:target_recurrence}
    P_U(\mathbf{x}) = \sum_{\substack{\mathbf T \geq\mathbf 1\\ T \subsetneq U}} (-1)^{|U\setminus T|-1} P_T(\mathbf{x}) \mathbf{e}_1(T)^{\mathsf{A}(Q)(\mathbf U -\mathbf T )}.
\end{equation}
\end{remark}

\subsection{Cardinality of the set of eventually constant representations}
\label{subsec:card_EC}

In this section, we derive a matrix enumeration formula for the cardinality of the set of eventually constant representations. We pass from the attaching weight matrix $M$ to the compressed attaching cardinality matrix $N$. We use the notation and results in Sections~\ref{subsec:multinomial_notation} and~\ref{subsec:gen_function_vertex}.

Consider the homomorphism $R_{\overline V}\to\mathbb Z$ induced by specializing all variables to $1$, $x_v\mapsto 1$. Since each subgraph of $G(Q)$ is assigned a monomial, the polynomial weights of each subset of graphs $\mathcal A(T,U)$, $\mathcal R(S)$, $\mathcal D_S(U)$, $\EC_S(U)$, and $\EC$ evaluate to their cardinalities.

For all subsets $S,\,T\subseteq U\subseteq \overline V$, the attaching set has size
\begin{equation}\label{quiver_attaching_cardinality}
    |\mathcal A(U\setminus T,U)|=\prod_{b\in Q_0}\mathbf{T}_b^{(\mathsf A(Q)(\mathbf{U}-\mathbf{T}))_b}=\mathbf{T}^{\mathsf A(Q)(\mathbf{U}-\mathbf{T})}.
\end{equation}
Therefore, Corollary~\ref{Hyperforest enumerator chain formulation-Prop} yields
\begin{equation}\label{first_card_quiver}
       |\mathcal D_S(U)| = \sum_{\Gamma\in\mathcal{C}(S,U)}(-1)^{|U\setminus S|-L(\Gamma)}\prod_{\ell=1}^{L(\Gamma)} \bm{\Gamma}_{\ell-1}^{\mathsf A(Q)(\bm{\Gamma}_\ell-\bm{\Gamma}_{\ell-1})}.
\end{equation}
The contribution of each chain $\Gamma=(S,\Gamma_1,\ldots,\Gamma_{L(\Gamma)-1},U)$ in $\mathcal P(\overline V)$ only depends on the chain of cardinality vectors $(\mathbf{S},\bm{\Gamma}_1,\ldots,\bm{\Gamma}_{L(\Gamma)-1},\mathbf{U})$ in $[\mathbf{0},\mathbf{\overline V}]\subseteq\mathbb N^{Q_0}$. Grouping chains by their fiber under this map yields a sum over chains in $[\mathbf{0},\mathbf{\overline V}]$. In turn, this sum gives an enumeration formula with the attaching cardinality matrix $N$, a strictly upper triangular matrix indexed on $[\mathbf{0},\mathbf{\overline V}]$.

\begin{theorem}\label{thm_cardinality_quiver}
  Fix any subsets $S\subseteq U\subseteq \overline V$. The cardinality of $\mathcal D_S(U)$ is given by
\begin{equation}\label{cardinality_matrix_quiver}
      |\mathcal D_S(U)|=(\mathbf{U}-\mathbf{S})!(I-N)^{-1}_{\mathbf{S},\mathbf{U}}.
  \end{equation}
  As a sum over chains $\bm{\Gamma}=(\mathbf{S},\bm{\Gamma}_1,\ldots,\bm{\Gamma}_{L(\bm\Gamma)-1},\mathbf{U})$ in $\mathbb N^{Q_0}$, we also have
\begin{equation}\label{cardinality_chain_quiver}
      |\mathcal D_S(U)|=\sum_{\bm{\Gamma}\in\mathcal C_{\mathbb N^{Q_0}}(\mathbf{S},\mathbf{U})}(-1)^{|U\setminus S|-L(\bm\Gamma)}\binom{\mathbf{U}-\mathbf{S}}{\bm{\Gamma}_1-\mathbf{S},\ldots,\mathbf{U}-\bm{\Gamma}_{L(\bm\Gamma)-1}}\prod_{\ell=1}^{L(\bm\Gamma)} \bm{\Gamma}_{\ell-1}^{\mathsf A(Q)(\bm{\Gamma}_\ell-\bm{\Gamma}_{\ell-1})}.
  \end{equation}
\end{theorem}
\begin{proof}
Consider the map $\mathcal P(\overline V)\to[\mathbf{0},\mathbf{\overline V}]$ given by $T\mapsto\mathbf{T}$. Since this map is strictly increasing, it induces a map on chains, $\mathcal C_{\mathcal P(\overline V)}(S,U)\to\mathcal C_{\mathbb N^{Q_0}}(\mathbf{S},\mathbf{U})$, given by $(S,\Gamma_1,\ldots,\Gamma_{L(\Gamma)-1},U)\mapsto (\mathbf{S},\bm{\Gamma}_1,\ldots,\bm{\Gamma}_{L(\Gamma)-1},\mathbf{U}).$

Now, we fix any chain $\bm{\Gamma}=(\mathbf{S},\bm{\Gamma}_1,\ldots,\bm{\Gamma}_{d-1},\mathbf{U})$ in $\mathcal C_{\mathbb N^{Q_0}}(\mathbf{S},\mathbf{U})$ and calculate the size of its fiber. Consider any subchain $(S,\Gamma_1,\ldots,\Gamma_{\ell-1})$ of subsets that corresponds to subchain $(\bm{\Gamma}_0,\ldots,\bm{\Gamma}_{\ell-1})$. For each vertex $a\in Q_0$, there are $\binom{\mathbf{U}_a-(\bm{\Gamma}_{\ell-1})_a}{(\bm{\Gamma}_\ell)_a-(\bm{\Gamma}_{\ell-1})_a}$ ways to extend to an appropriate subset $(\Gamma_\ell)_a\subseteq U_a$. Taking the product over each vertex in $Q_0$, there are $\binom{\mathbf{U}-\bm{\Gamma}_{\ell-1}}{\bm{\Gamma}_\ell-\bm{\Gamma}_{\ell-1}}$ ways to choose a set $\Gamma_\ell\setminus \Gamma_{\ell-1}$ so that $(S,\Gamma_1,\ldots,\Gamma_\ell)$ extends to $(\bm{\Gamma}_0,\ldots,\bm{\Gamma}_{\ell})$. Consequently, the size of the fiber is
    \begin{equation*}
   \binom{\mathbf{U}-\mathbf{S}}{\bm{\Gamma}_1-\mathbf{S}}\binom{\mathbf{U}-\bm{\Gamma}_1}{\bm{\Gamma}_2-\bm{\Gamma}_1}\cdots\binom{\mathbf{U}-\bm{\Gamma}_{d-2}}{\bm{\Gamma}_{d-1}-\bm{\Gamma}_{d-2}}\binom{\mathbf{U}-\bm{\Gamma}_{d-1}}{\mathbf{U}-\bm{\Gamma}_{d-1}}=\binom{\mathbf{U}-\mathbf{S}}{\bm{\Gamma}_1-\mathbf{S},\ldots,\mathbf{U}-\bm{\Gamma}_{d-1}}.
    \end{equation*}

Grouping chains of subsets by their fiber under this map, we have
    \begin{equation*}
    \begin{split}
     &\sum_{\Gamma\in\mathcal C_{\mathcal P(\overline V)}(S,U)}\prod_{\ell=1}^{L(\Gamma)} \bm{\Gamma}_{\ell-1}^{\mathsf A(Q)(\bm{\Gamma}_\ell-\bm{\Gamma}_{\ell-1})} \\
     &\qquad= \sum_{\bm{\Gamma}\in\mathcal C_{\mathbb N^{Q_0}}(\mathbf{S},\mathbf{U})}\binom{\mathbf{U}-\mathbf{S}}{\bm{\Gamma}_1-\mathbf{S},\ldots,\mathbf{U}-\bm{\Gamma}_{L(\bm\Gamma)-1}}\prod_{\ell=1}^{L(\bm\Gamma)} \bm{\Gamma}_{\ell-1}^{\mathsf A(Q)(\bm{\Gamma}_\ell-\bm{\Gamma}_{\ell-1})}.
    \end{split}
    \end{equation*}
 Then,~\eqref{cardinality_chain_quiver} follows from~\eqref{first_card_quiver}. Finally, we can factor each coefficient in~\eqref{cardinality_chain_quiver} by
      \begin{equation}\label{factorizing_quiver_coefficient}
          (- 1)^{|U\setminus S|-L(\bm\Gamma)}\binom{\mathbf{U}-\mathbf{S}}{\bm{\Gamma}_1-\mathbf{S},\ldots,\mathbf{U}-\bm{\Gamma}_{L(\bm\Gamma)-1}}=(\mathbf{U} -\mathbf{S})!\prod_{\ell=1}^{L(\bm\Gamma)}\frac{(- 1)^{|\bm{\Gamma}_\ell-\bm{\Gamma}_{\ell-1}|-1}}{(\bm{\Gamma}_\ell-\bm{\Gamma}_{\ell-1})!}.
      \end{equation}
Decomposing each coefficient in~\eqref{cardinality_chain_quiver} by~\eqref{factorizing_quiver_coefficient} gives
\begin{equation*}
        |\mathcal D_S(U)|=(\mathbf{U}-\mathbf{S})!\sum_{\bm{\Gamma}\in\mathcal C_{\mathbb N^{Q_0}}(\mathbf{S},\mathbf{U})}\prod_{\ell=1}^{L(\bm\Gamma)}N_{\bm{\Gamma}_{\ell-1},\bm{\Gamma}_{\ell}},
    \end{equation*}
which is equivalent to~\eqref{cardinality_matrix_quiver} by applying the matrix-chain equivalence given in~\eqref{matrix-chain-equivalence}.
\end{proof}

Consider any vertex subsets $S\subseteq U\subseteq \overline V$. The cardinality of the set of representations on $S$ is $|\mathcal R(S)|=\prod_{b\in Q_0}\mathbf{S}_b^{(\mathsf A(Q)\mathbf{S})_b}=\mathbf{S}^{\mathsf A(Q)\mathbf{S}}$.
Thus, Theorem~\ref{thm_cardinality_quiver} and Lemma~\ref{lem:representations_decompose} give
\begin{equation}\label{eq:cardinality_independence}
    |\EC_S(U)|=|\mathcal R(S)||\mathcal D_S(U)|=\mathbf{S}^{\mathsf A(Q)\mathbf{S}}(\mathbf{U}-\mathbf{S})!(I-N)^{-1}_{\mathbf{S},\mathbf{U}}.
\end{equation}
In particular, observe that $|\EC_S(U)|$ depends only on the cardinality vectors $\mathbf{S}$ and $\mathbf{U}$.

Recall that $\mathbf{1} = (1,\ldots, 1)$ is a $Q_0$-indexed vector.
Fix any transversal $C\subseteq \overline V$. Then $\mathbf{C}=\mathbf{1}$, $|\mathcal R(C)|=1$, and $|\EC_C(\overline V)|=(\mathbf{\overline V}-\mathbf{1})!(I-N)^{-1}_{\mathbf{1},\mathbf{\overline V}}$.
Also, observe that there are $\prod_{a\in Q_0}|V_a|=\mathbf{\overline V}^{\mathbf{1}}$ many transversals on $\overline V$. Since $\mathbf{\overline V}^{\mathbf{1}}(\mathbf{\overline V}-\mathbf{1})!=\mathbf{\overline V}!$, we have the following theorem.

\begin{theorem}\label{thm:EC_cardinality_matrix_enumerator}
The cardinality of the set of eventually constant set-valued representations is
\begin{equation}
    |\EC|=\mathbf{\overline V}!(I-N)^{-1}_{\mathbf{1},\mathbf{\overline V}}.
\end{equation}
Stated in terms of chains in $\mathbb N^{Q_0}$, the cardinality is given by
\begin{equation}
    |\EC|=\sum_{\bm{\Gamma}\in\mathcal C_{\mathbb N^{Q_0}}(\mathbf{1},\mathbf{\overline V})}(-1)^{|\overline V|-|Q_0|-L(\bm\Gamma)}\binom{\mathbf{\overline V}}{\mathbf{1}, \bm{\Gamma}_1-\mathbf{1},\ldots,\mathbf{\overline V}-\bm{\Gamma}_{L(\bm\Gamma)-1}}\prod_{\ell=1}^{L(\bm\Gamma)} \bm{\Gamma}_{\ell-1}^{\mathsf A(Q)(\bm{\Gamma}_\ell-\bm{\Gamma}_{\ell-1})}.
\end{equation}
\end{theorem}

\begin{remark}\label{remark:card_recurrence}
    As observed above, the cardinality $|\EC_S(U)|$ depends only on the cardinality vectors $\mathbf{S}$ and $\mathbf{U}$. So, we can define the function
    \[
    E:\{(\mathbf c,\mathbf d)\in[\mathbf{0},\mathbf{\overline V}]^2 : \mathbf c\leq\mathbf d\} \longrightarrow \mathbb{N},
    \]
    by $E(\mathbf{c}, \mathbf{d}) := |\EC_S(U)|$, for any vertex subsets $S\subseteq U\subseteq \overline V$ such that $\mathbf{S}=\mathbf{c}$ and $\mathbf{U}=\mathbf{d}$.

    The recurrence for $(I-M)^{-1}$ in~\eqref{general_recursion_formula} extends to $(I-N)^{-1}$ since $N$ is strictly upper triangular. So for all $\mathbf c\leq\mathbf d\in[\mathbf{0},\mathbf{\overline V}]$, the equation in~\eqref{eq:cardinality_independence} yields the recurrence
    \begin{equation}
        E(\mathbf{c}, \mathbf{d}) 
        =
        \mathbf c ^{(\mathsf A(Q)\mathbf c)}
        \delta_{\mathbf c,\mathbf d}+ \sum_{\mathbf{c}\leq \mathbf{t} < \mathbf{d}} (-1)^{|\mathbf{d}-\mathbf{t}| - 1} \binom{\mathbf{d}-\mathbf{c}}{\mathbf{t}-\mathbf{c}} 
        \mathbf{t}^{\mathsf A(Q)(\mathbf{d}-\mathbf{t})} E(\mathbf{c}, \mathbf{t}).
    \end{equation}
    Since there are $\mathbf{\overline V}^{\mathbf{1}}$ transversals in $\overline V$, evaluating $E(\mathbf{1}, \mathbf{\overline V})$ gives
    \begin{equation}
        |\EC|=\delta_{\mathbf{1},\mathbf{\overline V}}+\mathbf{\overline V}^{\mathbf{1}} \sum_{\mathbf{1} \leq \mathbf{t} < \mathbf{\overline V}} (-1)^{|\mathbf{\overline V}-\mathbf{t}| - 1} \binom{\mathbf{\overline V}-\mathbf{1}}{\mathbf{t}-\mathbf{1}} \mathbf{t}^{\mathsf A(Q)(\mathbf{\overline V}-\mathbf{t})} E(\mathbf{1}, \mathbf{t}).
    \end{equation}
\end{remark}

\section{Eventually constant representations for Jordan and cyclic quivers}
\label{sec:word_eventually_constant}
We study particular indexing quivers $Q$. We analyze their eventually constant representations and associated generating polynomials, whose general forms we derived in Section~\ref{sec:eventually_constant_set-valued_representations}.

\subsection{Eventually constant endomorphisms over Jordan quivers}
Fix a positive integer $j\in\mathbb N$. In this section, let $Q$ be a quiver with a single vertex $Q_0:=\{v_0\}$ and $j$ loops, denoted $1,\ldots,j$. Fix a single nonempty finite set $V_{v_0}$. We use the notation given in Sections~\ref{subsection_quivers},~\ref{subsec:multinomial_notation}, and the results in Section~\ref{sec:eventually_constant_set-valued_representations}.

Since $Q_0$ has a single vertex, we have $\coprod_{a\in Q_0}V_a=V_{v_0}$, which we simply denote as $V$. Each representation $W\in\Rep(Q,V)$ identifies with the $j$-tuple of endomorphisms $W:=(W_1,\ldots,W_{j})\in\End(V)^j$. See Figure~\ref{fig_00015}.

\begin{figure}[htbp]
    \centering
    \begin{tikzpicture}[scale=1, transform shape, baseline=(current bounding box.center),
        quiver_node/.style={circle, draw=black!70, fill=orange!15, very thick, minimum size=8.5mm, inner sep=0.5pt, font=\large},
        edge_r/.style={->, >={Stealth[length=2.5mm]}, red!75!black, very thick, dotted},
        edge_b/.style={->, >={Stealth[length=2.5mm]}, blue!75!black, very thick},
        edge_g/.style={->, >={Stealth[length=2.5mm]}, green!50!black, very thick, dash pattern=on 5pt off 2pt}
        ]
        
        \node[quiver_node] (v0) at (0, 0) {$v_0$};
        
        \draw[edge_b] (v0) edge[out=120, in=60, looseness=8, min distance=18mm] node[midway, above] {$1$} (v0);
        
        \draw[edge_g] (v0) edge[out=360, in=300, looseness=8, min distance=18mm] node[midway, right=1mm] {$2$} (v0);
        
        \draw[edge_r] (v0) edge[out=240, in=180, looseness=8, min distance=18mm] node[midway, left=1mm] {$3$} (v0);
        
    \end{tikzpicture}
    \hspace{1.5cm}
    \begin{tikzpicture}[scale=0.9, transform shape, baseline=(current bounding box.center),
        domain_node/.style={circle, draw=black!70, fill=orange!15, very thick, minimum size=8.5mm, inner sep=0.5pt, font=\large},
        root_node/.style={circle, draw=MidnightBlue, fill=blue!15, very thick, minimum size=8.5mm, inner sep=0.5pt, font=\large},
        edge_r/.style={->, >={Stealth[length=2.5mm]}, red!75!black, very thick, dotted},
        edge_b/.style={->, >={Stealth[length=2.5mm]}, blue!75!black, very thick},
        edge_g/.style={->, >={Stealth[length=2.5mm]}, green!50!black, very thick, dash pattern=on 5pt off 2pt}
        ]
        
        \node[root_node] (n7) at (3, .25) {7};
        
        \node[domain_node] (n1) at (3.0, 4.8) {1}; 
        \node[domain_node] (n2) at (1.5, 3.5) {2}; 
        \node[domain_node] (n3) at (4.5, 3.5) {3}; 
        \node[domain_node] (n4) at (0.5, 2.0) {4}; 
        \node[domain_node] (n5) at (3.0, 2.5) {5}; 
        \node[domain_node] (n6) at (5.5, 2.0) {6}; 
        
        \draw[edge_r] (n1) edge[bend left=15] (n2);
        \draw[edge_r] (n2) to (n4);
        \draw[edge_r] (n3) to (n6);
        \draw[edge_r] (n4) to (n7);
        \draw[edge_r] (n6) to (n5);
        \draw[edge_r] (n5) to (n7);
        \draw[edge_r] (n7) edge[loop left, looseness=6, min distance = 15mm] (n7);
        
        \draw[edge_b] (n1) edge[bend left=15] (n3);
        \draw[edge_b] (n2) edge[bend left=15] (n5);
        \draw[edge_b] (n3) edge[bend left=15] (n7);
        \draw[edge_b] (n4) edge[bend left=15] (n5);
        \draw[edge_b] (n6) edge[bend left=12] (n7);
        \draw[edge_b] (n5) edge[bend left=15] (n7);
        \draw[edge_b] (n7) edge[loop right, looseness=6, min distance = 15mm] (n7);
        
        \draw[edge_g] (n1) edge[bend right=15] (n2);
        \draw[edge_g] (n2) edge[bend right=15] (n7);
        \draw[edge_g] (n3) edge[bend right=15] (n5);
        \draw[edge_g] (n4) edge[bend right=15] (n6);
        \draw[edge_g] (n6) edge[bend right=02] (n7);
        \draw[edge_g] (n5) edge[bend right=15] (n7);
        \draw[edge_g] (n7) edge[loop below, looseness=6, min distance = 15mm] (n7);
    \end{tikzpicture}
    
    \caption{On the left is the $3$-Jordan quiver $Q$ with a single vertex $v_0$ and three loops $1, 2, 3$. On the right is the graph $G(W)$ of an eventually constant representation $W=(W_1,W_2,W_3)\in\Rep(Q,[7])$, stabilizing to the vertex $\{7\}$. The thick blue arrows are the $1$-labeled edges $(u,v)^{(1)}$ corresponding to the map $W_{1}$. Similarly, the dashed green arrows correspond to $W_{2}$, and the dotted red arrows correspond to $W_{3}$.}
    \label{fig_00015}
\end{figure}
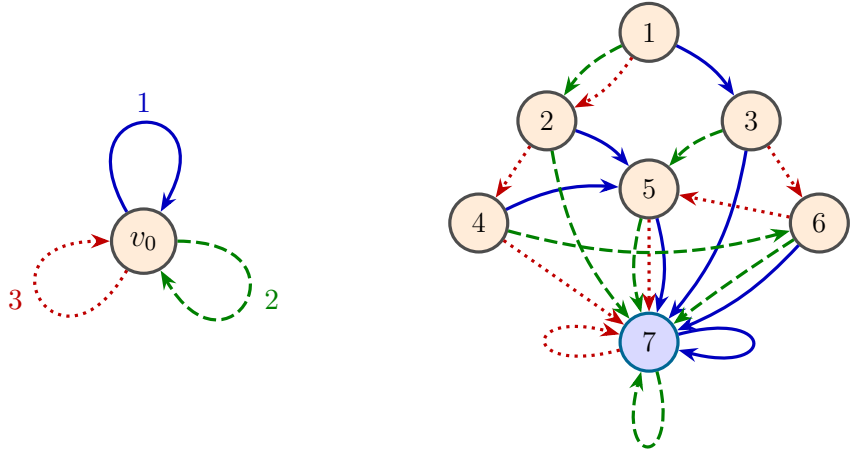

All loops in $Q$ are composable, so all words generated by the loops yield valid paths $i_\ell i_{\ell-1}\cdots i_2i_1$ in $Q$.
Therefore, a representation $W$ stabilizes in a vertex subset $S\subseteq V$ if $S$ is closed under each endomorphism $W_i$ and all sufficiently long words $W_{i_\ell}\cdots W_{i_1}$ in the subsemigroup generated by $\{W_1,\ldots,W_j\}$ in $\End(V)$ map $V$ into $S$.
In particular, a transversal $C\subseteq V$ is a singleton set, $\{c\}\subseteq V$. So, a representation $W$ is eventually constant if all sufficiently long words $W_{i_\ell}\cdots W_{i_1}$ generated by the endomorphisms $\{W_1,\ldots,W_j\}$ become a single constant map.

Between each pair of vertices $u,v\in V$, the graph $G(Q)$ has $j$ directed edges from $u$ to $v$, $(u,v)^{(1)},\ldots,(u,v)^{(j)}$. Each representation $W$ identifies with the graph $G(W)$ on $V$ whose edges are given by the disjoint union,
\[
\{(u,W_1(u))^{(1)} : u\in V\}\sqcup \ldots \sqcup\{(u,W_j(u))^{(j)} : u\in V\}.
\]
The representation $W$ has the weight
\[w(G(W))=\prod_{i=1}^j\prod_{u\in V} x_{u,W_i(u)}^{(i)},\quad w(G(W))=\prod_{i=1}^j\prod_{u\in V} x_{W_i(u)}\]
in edge and target variables, respectively.

We now apply the enumerators from Section~\ref{sec:eventually_constant_set-valued_representations} to the $j$-Jordan quiver. Specializing~\eqref{eq_W_quiver_full_var}, the entries of the attaching weight matrix $M$ are
\begin{equation}
    M_{T,U}=
    \begin{cases}
    (-1)^{|U\setminus T|-1}\prod_{i=1}^j\prod_{u\in U\setminus T}\sum_{v\in T} x_{u,v}^{(i)}&\text{if } T\subsetneq U,\\
    \hspace{32mm} 0 &\text{otherwise.}
    \end{cases}
\end{equation}
We give two forms for the generating function, as a sum over the matrix entries $(I-M)^{-1}_{\{c\},V}$ and its expansion over chains of subsets,
\begin{equation}
\begin{aligned}
    P(\mathbf x) &=\sum_{c\in V} \left(\prod_{i=1}^j x_{c,c}^{(i)}\right)(I-M)^{-1}_{\{c\},V} \\
    &=\sum_{c\in V} \left(\prod_{i=1}^j x_{c,c}^{(i)}\right) \sum_{\Gamma\in\mathcal{C}(\{c\},V)}(-1)^{|V|-1-L(\Gamma)}\prod_{\ell=1}^{L(\Gamma)}\prod_{i=1}^j\prod_{u\in \Gamma_\ell\setminus \Gamma_{\ell-1}}\sum_{v\in \Gamma_{\ell-1}}x_{u,v}^{(i)}.
\end{aligned}
\end{equation}

The adjacency matrix of $Q$ is the $1 \times 1$ matrix $\mathsf A(Q)$ with entry $j$, and for each subset $U\subseteq V$, the cardinality vector is simply $|U|$. Thus, $(\mathsf A(Q) \mathbf U)=j|U|$. So, the target-specialized transition weight matrix $M$ has entries
\begin{equation}\label{target_weight_simplification}
    M_{T,U}=
    \begin{cases}
    (-1)^{|U\setminus T|-1} e_1(T)^{j|U\setminus T|}&\text{if } T\subsetneq U,\\
    \hspace{19mm} 0&\text{otherwise.}
    \end{cases}
\end{equation}
The target-specialized generating function is
\begin{equation}
    P(\mathbf x) = \sum_{c\in V} x_c^j \sum_{\Gamma\in\mathcal{C}(\{c\},V)}(-1)^{|V|-1-L(\Gamma)}\prod_{\ell=1}^{L(\Gamma)}e_1(\Gamma_{\ell-1})^{j|\Gamma_\ell\setminus \Gamma_{\ell-1}|}.
\end{equation}
In Section~\ref{subsec:gen_function_vertex}, we established that for a general quiver $Q'$, the generating function is invariant under permutations of each set in the $Q'_0$-indexed tuple $(V_a)_{a\in Q'_0}$. Since $Q$ has only a single vertex, $P(\mathbf{x})$ is invariant under the symmetric group $\mathfrak{S}_V$. Thus, $P(\mathbf{x})$ is a symmetric polynomial in the variables $\{x_v : v \in V\}$.

\begin{example}
    Suppose $Q$ is the Jordan quiver with $2$ loops and let $V:=[3]$. In target variables, the submatrix $M[\{1\},[3]]$ is given by
    \[\bordermatrix{
      & \{1\} & \{1,2\} & \{1,3\} & [3] \cr
    \:\:\{1\} & 0 & x_1^2 & x_1^2 & -x_1^4 \cr
    \{1,2\} & 0 & 0 & 0 & (x_1+x_2)^2 \cr
    \{1,3\} & 0 & 0 & 0 & (x_1+x_3)^2 \cr
    \:\:\: [3] & 0 & 0 & 0 & 0 \cr}.\]
    We then see that
    \[w(\EC_{\{1\}}([3])) = x_1^2 \left( -x_1^4 + x_1^2 (x_1+x_2)^2 + x_1^2 (x_1+x_3)^2 \right).\]
   We can evaluate the weight for the remaining transversals by symmetry. This yields the generating function
    \begin{align*}
    P(\mathbf x) &= x_1^2 \left( -x_1^4 + x_1^2(x_1+x_2)^2 + x_1^2(x_1+x_3)^2 \right) \\
    &\quad + x_2^2 \left( -x_2^4 + x_2^2(x_2+x_1)^2 + x_2^2(x_2+x_3)^2 \right) \\
    &\quad + x_3^2 \left( -x_3^4 + x_3^2(x_3+x_1)^2 + x_3^2(x_3+x_2)^2 \right).
    \end{align*}

    Suppose instead that $V:=[100]$. The lattice of sets between $[98]$ and $[100]$ is isomorphic to the lattice between $[1]$ and $[3]$, so the calculation of $w(\EC_{[98]}([100]))$ follows similarly. Since $w(\mathcal{R}([98])) = e_1([98])^{196}$, we have
    \[w(\EC_{[98]}([100])) = e_1([98])^{196} \left( -e_1([98])^4 + e_1([98])^2 e_1([99])^2 + e_1([98])^2 e_1([98]\cup\{100\})^2 \right). \]
\end{example}

Finally, we evaluate the cardinality of the eventually constant representations. Since vectors in $\mathbb N^{Q_0}$ reduce to scalars $t,\,u\in\mathbb N$, the attaching cardinality matrix $N$ is indexed on the integer interval $[0,|V|]$. Its entries are given by
\begin{equation}
    N_{t,u}=
    \begin{cases}
        \frac{(-1)^{u-t-1}}{(u-t)!}t^{j(u-t)}&\text{if }t< u,\\
        \hspace{12mm} 0&\text{otherwise.}
    \end{cases}
\end{equation}
By Theorem~\ref{thm:EC_cardinality_matrix_enumerator}, the total number of eventually constant $j$-tuples of endomorphisms on $V$ is
\begin{equation}
    |\EC|=|V|!(I-N)^{-1}_{1,|V|}.
\end{equation}
For any subset $U\subseteq V$,~\eqref{eq:cardinality_independence} gives
\[|\EC_S(U)|=|S|^{j|S|}(|U|-|S|)!(I-N)^{-1}_{|S|,|U|}.\]

\begin{example}
    Suppose $Q$ is the Jordan quiver with $3$ loops and let $V:=[5]$. Then the submatrix $N[1,5]$ is
    \[
    \begin{pmatrix}
    0 & 1 & -1/2 & 1/6 & -1/24 \\
    0 & 0 & 8 & -32 & 256/3\\
    0 & 0 & 0 & 27 & -729/2\\
    0 & 0 & 0 & 0 & 64 \\
    0 & 0 & 0 & 0 & 0
    \end{pmatrix}.
    \]
    The number of eventually constant representations of $Q$ over $V$ is
    \[|\EC| = 5! (I-N)^{-1}_{1,5} = 992{,}905.\]
\end{example}
\begin{remark}
    Recall the function $E(\mathbf{c}, \mathbf{d})$ defined in Remark~\ref{remark:card_recurrence} for the cardinality $|\EC_S(U)|$. Since the adjacency matrix is $\mathsf{A}(Q) = [j]$ and the subset cardinality vectors collapse to standard cardinality, the recurrence for $E(s,u)$ specializes to
    \begin{equation}
        E(s,u) = \delta_{s,u} s^{js} + \sum_{s \leq t < u} (-1)^{u-t-1} \binom{u-s}{t-s} t^{j(u-t)} E(s,t).
    \end{equation}
    Thus, the cardinality of the set of eventually constant representations is given by
    \begin{equation}
        |\EC|= \delta_{1,|V|} + |V|\sum_{1 \leq t < |V|} (-1)^{|V|-t-1} \binom{|V|-1}{t-1} t^{j(|V|-t)} E(1,t).
    \end{equation}

    This specialization recovers the linear recurrence for the cardinality of the set of labeled quasi-acyclic automata derived by Liskovets~\cite[Theorem 3.1]{Ls06} under the algebraic substitution $k=j$, $r=s$, $n=u-s$, and scaling by the sink weight $E(s,n+s) = a_j(n,s)s^{js}$.
\end{remark}

\subsection{Eventually constant representations over cyclic quivers}
\label{sec:cycle_eventually_constant}

In this section, we consider the eventually constant representations on a cyclic quiver, as studied in~\cite{GHI26,CIKLR25,CILR25}. We use the notation established in Sections~\ref{sec:eventually_constant_set-valued_representations},~\ref{subsection_quivers}, and~\ref{subsec:multinomial_notation}.

Fix an integer $k\in\mathbb N$. Let $Q$ be the cyclic quiver on $k$ vertices indexed by $\mathbb Z/k\mathbb Z$, with arrows given by
\[1\xrightarrow {\alpha_1}2\xrightarrow{\alpha_2}\ldots \xrightarrow{\alpha_{k-2}} k-1 \xrightarrow{\alpha_{k-1}}k\xrightarrow{\alpha_k}1,\]
all understood modulo $k$. For clarity, we write each arrow $\alpha_i$ as $i$.
Also, fix $k$ disjoint finite sets $V_1,\ldots,V_k$, and let $\overline V:=\coprod_{i=1}^k V_i$.

Each representation $W\in\Rep(Q, V)$ identifies with a $k$-tuple $W:=(W_1,\ldots,W_k)$ of functions directed from each set $V_i$ to the next,
\[V_1\xrightarrow {W_1}V_2\xrightarrow{W_2}\ldots\xrightarrow{W_{k-1}}V_k\xrightarrow {W_k}V_1.\]
Transversals $C\subseteq\overline V$ are simply $k$-tuples, $(c_1,\ldots,c_k)$, with each $c_i\in V_i$.
A representation $W$ is eventually constant if all sufficiently long cyclic compositions $W_{a+\ell}\circ \cdots \circ W_{a+1}\circ W_a$ stabilize in a transversal $(c_1,\ldots,c_k)$. See Figure~\ref{fig_00020}. Equivalently, $W$ is eventually constant if and only if for some $i=1,\ldots,k$, the endomorphism $W_{i+k-1}\circ\cdots \circ W_{i+1}\circ W_i:V_i\to V_i$ is eventually constant in the classical sense. In particular, if some $V_i$ is a singleton, then all representations in $\Rep(Q, V)$ are eventually constant.

The adjacency matrix $\mathsf A(Q)\in\mathbb Z^{k\times k}$ has entries
\[\mathsf A(Q)_{i,j}=
\begin{cases}
    1&\text{if }j \equiv i-1  \!\!\!\!\mod k,\\
    0&\text{otherwise}.
\end{cases}\]
Therefore, for each $k$-tuple $\mathbf c\in\mathbb N^{k}$,
\begin{equation}\label{def:shifted_k-tuple}
\mathsf A(Q)\mathbf c=(c_k,c_1,c_2,\ldots,c_{k-2},c_{k-1}).
\end{equation}
We denote this shifted $k$-tuple as $\mathbf c^-:=(c_k,c_1,c_2,\ldots,c_{k-2},c_{k-1})$.

Consider the target-specialized recursion given in Remark~\ref{rmk:target_recursion} for generating functions over subsets $U\subseteq\overline V$. If $U:=(u_1,\ldots,u_k)$ is a transversal, then
\[
P_U(\mathbf x)=\prod_{i=1}^kx_{u_i}=\mathbf e_1(U)^\mathbf 1,
\]
where $\mathbf{1} = (1,\ldots, 1)\in \mathbb{N}^{Q_0}$.
If $U$ is strictly larger than a transversal, then by~\eqref{def:shifted_k-tuple}, the recursive formula~\eqref{eq:target_recurrence} becomes
\begin{equation}\label{eq:recursive_formula_cyclic_quiver}
    P_U(\mathbf{x}) = \sum_{\substack{\mathbf T \geq\mathbf 1\\ T \subsetneq U}} (-1)^{|U\setminus T|-1} P_T(\mathbf{x}) \mathbf{e}_1(T)^{\mathbf U^- -\mathbf T^-}.
\end{equation}
Using this recursive formula, we derive the closed form for the generating function given in~\cite[Theorem 1.1]{GHI26}. This method avoids the use of the matrix-tree theorem.

We first modify our conventions. Instead of working in the polynomial ring $R_{\overline V}=\mathbb Z[x_v : v\in\overline V]$, we pass to its field of fractions $\Frac(R_{\overline V})$. In this field, if we have a $k$-tuple $\mathbf c\in\mathbb Z^k$ instead of $\mathbb N^k$, we can extend vector exponentiation with a $k$-tuple $\mathbf q\in R_{\overline V}^k$ in the natural way. If no term $\mathbf q_i$ is zero, we have
\[\mathbf q^{\mathbf c}:=\prod_{i=1}^k\mathbf q_i^{\mathbf c_i}\in \Frac(R_{\overline V}).\]
Also, we extend the notation for $\mathbf e_1$ to the second elementary symmetric polynomial. For any subset $S\subseteq V_i$, define the ordinary second elementary symmetric polynomial,
\[e_2(S):=\frac{1}{2}\sum_{\substack{u,w\in S\\u\neq w}}x_u x_w\in \Frac(R_{\overline V}).\]
Then for each subset $U\subseteq{\overline V}$, define the $k$-tuple
\[\mathbf e_2(U):=(e_2(U_1),\ldots,e_2(U_k))\in \Frac(R_{\overline V}).\]

\begin{proposition}
    The generating function over the $k$-cyclic quiver on $V$ is given by
    \begin{equation}
        P(\mathbf x)=\mathbf e_1(\overline V)^{\mathbf{\overline V}^-}-2^k\mathbf e_1(\overline V)^{\mathbf{\overline V}^--\mathbf 2}\mathbf e_2(\overline V)^\mathbf 1.
    \end{equation}
\end{proposition}
\begin{proof}
    We induct on the cardinality $|U|$, for subsets $U\subseteq\overline V$ containing a transversal. For each subset $U$, the assertion becomes
    \begin{equation}\label{eq:inductive_step_cyclic_recursion}
        P_U(\mathbf x)=\mathbf e_1(U)^{\mathbf{U}^-}-2^k\mathbf e_1(U)^{\mathbf{U}^--\mathbf 2}\mathbf e_2(U)^\mathbf 1.
    \end{equation}
Write the right-hand side as $F(U):=\mathbf e_1(U)^{\mathbf{U}^-}-2^k\mathbf e_1(U)^{\mathbf{U}^--\mathbf 2}\mathbf e_2(U)^\mathbf 1$. Fix any $U\subseteq \overline V$ that contains a transversal and for each $i=1,\ldots,k$, set $|U_i|:=n_i$. Since we identify subscripts modulo $k$, recall that $n_0=n_k$. Assume that~\eqref{eq:inductive_step_cyclic_recursion} holds for all proper subsets $T\subseteq U$ with $\mathbf T\geq\mathbf 1$.
    
    First suppose $\mathbf U\not\geq\mathbf 2$. Then for some vertex $i\in Q_0$, $|U_i|=1$, so $e_2(U_i)=0$ and thus $\mathbf e_2(U)^\mathbf 1=0$. Also, all representations in $\Rep(Q,U)$ are eventually constant, so~\eqref{R_U(S,T)_weight_targets_2} gives
    \[P_U(\mathbf x)=w(\mathcal R(U))=\mathbf e_1(U)^{\mathbf{U}^-}=F(U).\]
   
    Now, suppose that $\mathbf U\geq\mathbf 2$. By~\eqref{eq:recursive_formula_cyclic_quiver}, we have
    \begin{equation}\label{eq:cyclic_recursion_first}
        P_U(\mathbf x)
        =\sum_{\substack{\mathbf T \geq\mathbf 1\\ T \subsetneq U}} (-1)^{|U\setminus T|-1} \left(\mathbf e_1(T)^{\mathbf{U}^-}-2^k\mathbf e_1(T)^{\mathbf{U}^--\mathbf 2}\mathbf e_2(T)^\mathbf 1\right).
    \end{equation}
    This gives $P_U(\mathbf x)$ as a sum over proper subsets $T\subsetneq U$ with $\mathbf T\geq \mathbf 1$. To simplify, we pass to a sum over all subsets of $U$. First,~\eqref{eq:cyclic_recursion_first}
    still holds when we remove the condition $\mathbf T\geq \mathbf 1$. Since $\mathbf U^--\mathbf 2\geq\mathbf 0$, the polynomial $\mathbf e_1(T)^{\mathbf{U}^--\mathbf 2}$ is defined even if $\mathbf T\not\geq \mathbf 1$. Furthermore, for such $\mathbf T$, the corresponding summand $\mathbf e_1(T)^{\mathbf{U}^-}-2^k\mathbf e_1(T)^{\mathbf{U}^--\mathbf 2}\mathbf e_2(T)^\mathbf 1$ is $0$, since $\mathbf e_1(T)^{\mathbf{U}^-}$ and $\mathbf e_2(T)^\mathbf 1$ are $0$. Next, we can subtract $F(U)$ from~\eqref{eq:cyclic_recursion_first} to get
    \begin{equation}\label{eq:cyclic_recursion_equals_zero}
        P_U(\mathbf x)-F(U)=\sum_{T \subseteq U} (-1)^{|U\setminus T|-1} \left(\mathbf e_1(T)^{\mathbf{U}^-}-2^k\mathbf e_1(T)^{\mathbf{U}^--\mathbf 2}\mathbf e_2(T)^\mathbf 1\right).
    \end{equation}
    It remains to prove that this difference vanishes.

For each $i=1,\ldots,k$ and $S\subseteq U_i$, define the polynomials
\[A_i(S):=(-1)^{n_i-|S|} e_1(S)^{n_{i-1}},\quad\text{and}\quad B_i(S):=(-1)^{n_i-|S|} e_1(S)^{n_{i-1}- 2}2e_2(S).\]
For each $T\subseteq U$, the terms in~\eqref{eq:cyclic_recursion_equals_zero} expand as
\[(-1)^{|U\setminus T|-1}\mathbf e_1(T)^{\mathbf{U}^-}=-\prod_{i=1}^kA_i(T_i),\quad (-1)^{|U\setminus T|-1}\big(-2^k\mathbf e_1(T)^{\mathbf{U}^--\mathbf 2}\mathbf e_2(T)^\mathbf 1\big)=\prod_{i=1}^kB_i(T_i).\]
Since $\sum_{T_1 \subseteq U_1}\cdots\sum_{T_k\subseteq U_k}\prod_{i=1}^kA_i(T_i)=\prod_{i=1}^k\sum_{T_i\subseteq U_i}A_i(T_i)$ (and an analogous statement holds for $B_i(T_i)$), we can rewrite~\eqref{eq:cyclic_recursion_equals_zero} as
\begin{equation*}
P_U(\mathbf x)-F(U)=-\prod_{i=1}^k\sum_{T_i\subseteq U_i}A_i(T_i)+\prod_{i=1}^k\sum_{T_i\subseteq U_i}B_i(T_i).
\end{equation*}
So, the result reduces to proving that $\prod_{i=1}^k\sum_{T_i\subseteq U_i}A_i(T_i)=\prod_{i=1}^k\sum_{T_i\subseteq U_i}B_i(T_i)$.

We fix $i=1,\ldots,k$ and consider the sum over the $A_i(T_i)$ terms first. For each $T_i\subseteq U_i$, we can expand $e_1(T_i)^{n_{i-1}}$ as a sum over functions from $[n_{i-1}]$ to $T_i$, $\sum_{\varphi:[n_{i-1}]\to T_i}\prod_{j=1}^{n_{i-1}} x_{\varphi(j)}$. After, we can swap the order of summation to get
\begin{equation}\label{eq:A_i(T_i)_function_sum}
\begin{aligned}
\sum_{T_i\subseteq U_i}A_i(T_i)&=\sum_{T_i\subseteq U_i}\Big((-1)^{n_i-|T_i|}\sum_{\varphi:[n_{i-1}]\to T_i}\prod_{j=1}^{n_{i-1}} x_{\varphi(j)}\Big)\\
&=\sum_{\varphi:[n_{i-1}]\to U_i}\sum_{\im\varphi \subseteq T_i\subseteq U_i}(-1)^{n_i-|T_i|}\prod_{j=1}^{n_{i-1}} x_{\varphi(j)}.
\end{aligned}
\end{equation}
Here $\im\varphi$ is the image of the map. For each function $\varphi$, we have
\[\sum_{\im\varphi \subseteq T_i\subseteq U_i}(-1)^{n_i-|T_i|}=
\begin{cases}
    1&\text{if }\im\varphi=U_i,\\
    0&\text{otherwise.}
\end{cases}\]
Thus, if any $n_{j-1}<n_j$, then $\prod_{j=1}^k\sum_{T_j\subseteq U_j}A_j(T_j)=0$. So, this product can be nonzero only if $n_1\geq n_2\geq\ldots\geq n_k\geq n_1$, or equivalently, if all cardinalities $n_1,\ldots,n_k$ are equal. If so, then~\eqref{eq:A_i(T_i)_function_sum} is a sum over the bijections $\varphi:[n_i]\to U_i$, which gives $\sum_{T_i\subseteq U_i}A_i(T_i)=n_i!\prod_{v\in U_i}x_v$. Thus, $\prod_{j=1}^k\sum_{T_j\subseteq U_j}A_j(T_j)=(n_1!)^k\prod_{v\in U}x_v$.

Now, we fix $i=1,\ldots,k$ again and do the same for the sum over the $B_i(T_i)$ terms,
\begin{equation}\label{eq:B_i(T_i)_function_sum}
\begin{aligned}
\sum_{T_i\subseteq U_i}B_i(T_i)&=\sum_{T_i\subseteq U_i}\Big((-1)^{n_i-|T_i|}\sum_{\varphi:[n_{i-1}-2]\to T_i}\prod_{j=1}^{n_{i-1}-2} x_{\varphi(j)}\sum_{\substack{u,w\in T_i\\u\neq w}}x_ux_w\Big)\\
&=\sum_{\varphi:[n_{i-1}-2]\to U_i}\sum_{\substack{u,w\in U_i\\u\neq w}}\sum_{\im\varphi\cup\{u,w\} \subseteq T_i\subseteq U_i}(-1)^{n_i-|T_i|}x_ux_w\prod_{j=1}^{n_{i-1}-2} x_{\varphi(j)}.
\end{aligned}
\end{equation}
As before, the alternating sum vanishes unless $\im\varphi\cup\{u,w\}=U_i$. If $n_{j-1}<n_j$ for some $j$, then $\prod_{j=1}^k\sum_{T_j\subseteq U_j}B_j(T_j)=0$. If instead all cardinalities $n_1,\ldots,n_k$ are equal, then~\eqref{eq:B_i(T_i)_function_sum} is a sum over all functions $\varphi:[n_{i}-2]\to U_i$ and pairs $u,\,w\in U_i$ such that $\im\varphi\cup\{u,w\}=U_i$. These maps $\varphi$ are precisely the injections from $[n_{i}-2]$, with $\{u,w\}=U_i\setminus\im\varphi$. In these cases, $x_ux_w\prod_{j=1}^{n_{i}-2} x_{\varphi(j)}=\prod_{v\in U_i}x_v$. There are $(n_i-2)!\binom{n_i}{2}$ ways to choose an injection from $[n_{i}-2]$ into $U_i$, and there are two ways to order the remaining elements $u,\,w$, so $\sum_{T_i\subseteq U_i}B_i(T_i)=n_i!\prod_{v\in U_i}x_v$. Thus, $\prod_{j=1}^k\sum_{T_j\subseteq U_j}B_j(T_j)=(n_1!)^k\prod_{v\in U}x_v$.

Therefore, $P_U(\mathbf x)=F(U)$, and the result follows by induction.
\end{proof}

\section{Partition algorithm for the Jordan quiver}
\label{sec:partition_algorithm_Jordan_quiver}

We shift from matrix enumerators to outline an algorithm inspired by~\cite{Moore56,Hop71,Rut00,WDMS20,ABHKMS12}.
This gives an alternative method to recursively compute the generating function over Jordan quivers.
Instead of removing sources from acyclic graphs, we record the successive partitions of a finite set $V$ induced by identifying vertices mapped to the same output under compositions of increasing length.

Fix a set $V$ of cardinality $n\geq 2$ and a positive integer $j\in\mathbb N$. 
Let $Q$ be the Jordan quiver with $j$ loops, which are labeled by $1,\ldots,j$.

\subsection{Successor states}
\label{subsection_partition_label_map}

We work with partitions and equivalence relations on a finite set $V$. Elements of a partition $\P$ are called \defn{blocks}, and the block containing $u\in V$ is denoted $[u]_\P$. Note that $u\in V$ is in a unique block in a fixed partition.
More generally, if a subset $U\subseteq V$ is contained within a block of $\P$, let $[U]_\P$ denote the block containing $U$. Partitions are ordered so that $\P\leq\P'$ if $\P$ is \defn{finer} than $\P'$, meaning each block of $\P$ is contained in a unique block of $\P'$. We also say that $\P'$ is \defn{coarser} than $\P$. The finest partition is the \defn{discrete partition} $\P_0:=\{\{u\} : u\in V\}$, and the coarsest partition is $\{V\}$.

Fix an element $c\in V$. A \defn{state at $c$} is a tuple $\mathsf S=(\P,\P',\Phi)$ comprised of partitions $\P\leq\P'$ and an injective map $\Phi:\P'\to\P^j$, such that $\Phi([c]_{\P'}) = ([c]_{\P},\ldots,[c]_{\P})$. For a block $U\in\P'$, we denote the entries of the $j$-tuple $\Phi(U)$ as $(\Phi(U)_1,\ldots,\Phi(U)_j)$.
Let $\mathcal S(c)$ denote the set of states at $c$, and consider two of its subsets. Let $\mathcal S_I(c)\subseteq\mathcal S(c)$ denote the subset of \defn{initial states} $(\P_0,\P,\Phi)$ that start with the discrete partition $\P_0$, and let $\mathcal S_T(c)\subseteq\mathcal S(c)$ denote the subset of \defn{terminal states} $(\P,\{V\},\Phi)$ that end with the coarsest partition $\{V\}$.

Given partitions $\P\leq\P'$, there is a canonical surjective quotient map $\P\twoheadrightarrow\P'$ sending each block $U\in\P$ to $[U]_{\P'}$, the block in $\P'$ containing $U$. This map is surjective since each block of $\P'$ is a union of blocks of $\P$, and this is a quotient map since $\P'$ is obtained from $\P$ by identifying the blocks of $\P$ that lie in the same block of $\P'$. 
This induces a map on $j$-tuples $\P^j\to(\P')^j$. So if we have a state $\mathsf S=(\P,\P',\Phi)\in\mathcal S(c)$, the map $\Phi:\P'\to\P^j$ passes to a map $\overline\Phi:\P'\to(\P')^j$ by the composition
\[\overline\Phi(U):=([\Phi(U)_{1}]_{\P'},\ldots,[\Phi(U)_{j}]_{\P'}),\]
where $\Phi(U)_i\in\P$ is the $i$-th entry of $\Phi(U)$.
Next, we define a partition $\P''$ and map $\Phi':\P''\to(\P')^j$ so that we can form the \defn{successor} of $\mathsf S$, $\mathsf S'\in\mathcal S(c)$, where $\mathsf S':=(\P',\P'',\Phi')$. 
Construct $\P''$ by identifying elements $u$ and $v\in V$ if and only if
$\overline\Phi([u]_{\P'})=\overline\Phi([v]_{\P'})$. Equivalently, $\P''$ is induced from the fibers of the map
$u\mapsto\overline\Phi([u]_{\P'})$.
Then, define $\Phi':\P''\to(\P')^j$ so that for each $u\in V$,
\[\Phi'([u]_{\P''}):=\overline\Phi([u]_{\P'})=([\Phi([u]_{\P'})_{1}]_{\P'},\ldots,[\Phi([u]_{\P'})_{j}]_{\P'}).\]
The function is well-defined and injective since $[U]_{\P''}=[T]_{\P''}$ if and only if $\overline\Phi(U)=\overline\Phi(T)$. Also $\Phi'$ preserves $c$, since
\[\Phi'([c]_{\P''})=([[c]_\P]_{\P'},\ldots,[[c]_\P]_{\P'})=([c]_{\P'},\ldots,[c]_{\P'}).\]

\subsection{Recursive algorithm}

Fix an element $c\in V$ and let $\mathcal M(c)$ be the subset of representations $W\in\Rep(Q,V)$ whose maps $W_1,\ldots,W_j$ fix $c$.

Each representation $W\in\mathcal M(c)$ induces an initial state, $\mathsf S_1(W):=(\P_0,\P_1,\Phi_1)$. Form $\P_1$ by identifying vertices $u,v\in V$ whenever $W_i(u)=W_i(v)$ for each map $W_1,\ldots,W_j$, and define the map $\Phi_1:\P_1\to\P_0^j$ by $\Phi_1([u]_1):=(\{W_1(u)\},\ldots,\{W_j(u)\})$. Then $\Phi_1$ is injective by construction of $\P_1$, and since the maps of $W$ fix $c$, we have $\Phi_1([c]_1)=(\{c\},\ldots,\{c\})$.

The initial state $\mathsf S_1(W)$ induces a sequence of states at $c$ given by the successor relation, $(\mathsf S_m(W))_{m=1}^{n-1}$. For each $m < n-1$, define $\mathsf S_{m+1}(W):=\mathsf S_m(W)'$. We write $\mathsf S_m(W):=(\P_{m-1},\P_{m},\Phi_{m})$ and $[u]_m:=[u]_{\P_m}$.
By induction, we see that each partition $\P_m$ identifies vertices $u,v\in V$ if and only if $W_p(u)=W_p(v)$ for all paths $p=i_m\cdots i_1$ of length $m$. Also for each $u\in V$,
\[\Phi_{m}([u]_{m})=([W_1(u)]_{m-1},\ldots,[W_j(u)]_{m-1}).\]
Note that if $\P_{m-1}=\P_{m}$, then the sequence of partitions stabilizes since $\P_{m+1}=\P_{m}$; if $[u]_{m+1}=[v]_{m+1}$, then $[W_i(u)]_{m}=[W_i(v)]_{m}$ for each $i=1,\ldots,j$, so $[W_i(u)]_{m-1}=[W_i(v)]_{m-1}$ for each $i=1,\ldots,j$, so $[u]_{m}=[v]_{m}$.

The representation $W\in\mathcal M(c)$ is eventually constant if and only if the final state $\mathsf S_{n-1}(W)$ is terminal.
If $W\in\EC$, then all vertices $u$ are identified in $\P_{n-1}$ since for all $(n-1)$-length paths $p=i_{n-1}\cdots i_1$, $W_p(u)=c$. So, $\P_{n-1}=\{V\}$. If $W\not\in\EC$, then there is a path $p=i_{n-1}\cdots i_1$ for which $W_p$ is not the constant map at $c$. Thus for some $u\in V$, we have $W_p(u)\neq c$ and $W_p(c)=c$, so $[u]_{n-1}\neq[c]_{n-1}$. Consequently, $\P_{n-1}\neq\{V\}$, so $\mathsf S_{n-1}(W)$ is not terminal.

Next, note that the map $W\mapsto\mathsf S_1(W)$ from $\mathcal M(c)\to\mathcal S_I(c)$ is a bijection. For the inverse map, consider an initial state $(\P_0,\P,\Phi)\in\mathcal S_I(c)$. Since $\P_0$ is the discrete partition, we can define endomorphisms $W_1,\ldots,W_j\in\End(V)$ by $(\{W_1(u)\},\ldots,\{W_j(u)\}):=\Phi([u]_\P)$.
Each map $W_i$ fixes $c$, since $(\{W_1(c)\},\ldots,\{W_j(c)\})=\Phi([c]_\P)=(\{c\},\ldots,\{c\}).$

Now, we turn to target specialized weights and recursively construct maps
$\mu_1,\ldots,\mu_{n-1}$ that assign weights to states at $c$.
For each initial state $\mathsf S\in\mathcal S_I(c)$, let
\[\mu_1(\mathsf S):=w(W)=\prod_{i=1}^j\prod_{u\in V}x_{W_i(u)}\in R_V,\]
for $W$ the unique representation in $\mathcal M(c)$ corresponding to $\mathsf S$. Set $\mu_1$ to $0$ elsewhere.
Given the assignment $\mu_m$, define $\mu_{m+1}:\mathcal S(c)\to R_V$. For each state $\mathsf S$, let
\[\mu_{m+1}(\mathsf S):=\sum_{\substack{\mathsf T:\mathsf T'=\mathsf S}}\mu_m(\mathsf T).\]
By induction, the weight assignments admit another characterization.
For each $m\leq n-1$ and each state $\mathsf S\in\mathcal S(c)$, $\mu_m(\mathsf S)$ reduces to the weight
\[\mu_{m}(\mathsf S)=\sum_{\substack{W:\mathsf S_m(W)=\mathsf S}}w(W),\]
the weighted sum over all representations $W$ whose $m$-th successor state is $\mathsf S$.

Recall that $W\in\mathcal M(c)$ is eventually constant if and only if $\mathsf S_{n-1}(W)$ is a terminal state in $\mathcal S_T(c)$. Therefore, the weight of the eventually constant representations that fix $c$ is the following sum over the terminal states
\[\sum_{\mathsf S\in\mathcal S_T(c)}\mu_{n-1}(\mathsf S)=\sum_{W\in\EC\cap\mathcal M(c)}w(W).\]
Summing over all elements $c\in V$ then yields the generating function,
\[P(\mathbf x)=\sum_{c\in V}\sum_{\mathsf S\in\mathcal S_T(c)}\mu_{n-1}(\mathsf S).\]

\bibliographystyle{amsalpha} 
\bibliography{nilpotent_finite}

@article {Lei21,
    AUTHOR = {Leinster, Tom},
     TITLE = {The probability that an operator is nilpotent},
   JOURNAL = {Amer. Math. Monthly},
  FJOURNAL = {American Mathematical Monthly},
    VOLUME = {128},
      YEAR = {2021},
    NUMBER = {4},
     PAGES = {371--375},
}

@article {CIKLR25,
    AUTHOR = {Chen, Weixi and Im, Mee Seong and Khovanov, Mikhail and Lillja, Catherine and Rugo, Nicolas},
     TITLE = {Pairs of eventually constant maps and nilpotent pairs},
   JOURNAL = {arXiv preprint \href{https://arxiv.org/abs/2512.03367}{arXiv:2512.03367}, to appear in Lett. Math. Phys.},
      YEAR = {2025},
     PAGES = {1--15},
}

@article {CILR25,
    AUTHOR = {Chen, Weixi and Im, Mee Seong and Lillja, Catherine and Rugo, Nicolas},
     TITLE = {Eventually constant maps for two sets and nilpotent pairs},
   JOURNAL = {arXiv preprint \href{https://arxiv.org/abs/2512.05269}{arXiv:2512.05269}, to appear in Contemp. Math.},
      YEAR = {2025},
     PAGES = {1--21},
}

@incollection {Moore56,
    AUTHOR = {Moore, Edward F.},
     TITLE = {Gedanken-experiments on sequential machines},
 BOOKTITLE = {Automata studies},
    SERIES = {Ann. of Math. Stud.},
    VOLUME = {no. 34},
     PAGES = {129--153},
 PUBLISHER = {Princeton Univ. Press, Princeton, NJ},
      YEAR = {1956},
}

@article {GHI26,
    AUTHOR = {Green, Radford and Holmes, Cornell and Im, Mee Seong},
     TITLE = {Multisymmetric functions on eventually constant cyclic graphs},
   JOURNAL = {arXiv preprint \href{https://arxiv.org/abs/2604.16255}{arXiv:2604.16255}},
      YEAR = {2026},
     PAGES = {1--26},
}

@article {KR,
    AUTHOR = {Khovanov, Mikhail and Robert, Louis-Hadrien},
     TITLE = {Foam evaluation and {K}ronheimer-{M}rowka theories},
   JOURNAL = {Adv. Math.},
  FJOURNAL = {Advances in Mathematics},
    VOLUME = {376},
      YEAR = {2021},
     PAGES = {Paper No. 107433, 59},
}

@incollection {Q,
    AUTHOR = {Quillen, Daniel},
     TITLE = {Higher algebraic {$K$}-theory. {I}},
 BOOKTITLE = {Algebraic {$K$}-theory, {I}: {H}igher {$K$}-theories ({P}roc.
              {C}onf., {B}attelle {M}emorial {I}nst., {S}eattle, {W}ash.,
              1972)},
    SERIES = {Lecture Notes in Math.},
    VOLUME = {Vol. 341},
     PAGES = {85--147},
 PUBLISHER = {Springer, Berlin-New York},
      YEAR = {1973},
}

@article {CK78,
    AUTHOR = {Chaiken, Seth and Kleitman, Daniel J.},
     TITLE = {Matrix tree theorems},
   JOURNAL = {J. Combinatorial Theory Ser. A},
  FJOURNAL = {Journal of Combinatorial Theory. Series A},
    VOLUME = {24},
      YEAR = {1978},
    NUMBER = {3},
     PAGES = {377--381},
      ISSN = {0097-3165},
}

@book {St12,
    AUTHOR = {Stanley, Richard P.},
     TITLE = {Enumerative combinatorics. {V}olume 1},
    SERIES = {Cambridge Studies in Advanced Mathematics},
    VOLUME = {49},
   EDITION = {Second},
 PUBLISHER = {Cambridge University Press, Cambridge},
      YEAR = {2012},
     PAGES = {299--306},
}

@incollection {Rb73,
    AUTHOR = {Robinson, Robert W.},
     TITLE = {Counting labeled acyclic digraphs},
 BOOKTITLE = {New directions in the theory of graphs ({P}roc. {T}hird {A}nn
              {A}rbor {C}onf., {U}niv. {M}ichigan, {A}nn {A}rbor, {M}ich.,
              1971)},
     PAGES = {239--273},
 PUBLISHER = {Academic Press, New York-London},
      YEAR = {1973},
}

@article {Gs96,
    AUTHOR = {Gessel, Ira M.},
     TITLE = {Counting acyclic digraphs by sources and sinks},
   JOURNAL = {Discrete Math.},
  FJOURNAL = {Discrete Mathematics},
    VOLUME = {160},
      YEAR = {1996},
    NUMBER = {1-3},
     PAGES = {253--258},
}

@article {Ls06,
    AUTHOR = {Liskovets, Valery A.},
     TITLE = {Exact enumeration of acyclic deterministic automata},
   JOURNAL = {Discrete Appl. Math.},
  FJOURNAL = {Discrete Applied Mathematics. The Journal of Combinatorial
              Algorithms, Informatics and Computational Sciences},
    VOLUME = {154},
      YEAR = {2006},
    NUMBER = {3},
     PAGES = {537--551},
}

@article {Rt64,
    AUTHOR = {Rota, Gian-Carlo},
     TITLE = {On the foundations of combinatorial theory. {I}. {T}heory of
              {M}\"obius functions},
   JOURNAL = {Z. Wahrscheinlichkeitstheorie und Verw. Gebiete},
  FJOURNAL = {Zeitschrift f\"ur Wahrscheinlichkeitstheorie und Verwandte
              Gebiete},
    VOLUME = {2},
      YEAR = {1964},
     PAGES = {340--368},
}

@article {DW05,
    AUTHOR = {Derksen, Harm and Weyman, Jerzy},
     TITLE = {Quiver representations},
   JOURNAL = {Notices Amer. Math. Soc.},
  FJOURNAL = {Notices of the American Mathematical Society},
    VOLUME = {52},
      YEAR = {2005},
    NUMBER = {2},
     PAGES = {200--206},
}

@book {DW17,
    AUTHOR = {Derksen, Harm and Weyman, Jerzy},
     TITLE = {An introduction to quiver representations},
    SERIES = {Graduate Studies in Mathematics},
    VOLUME = {184},
 PUBLISHER = {American Mathematical Society, Providence, RI},
      YEAR = {2017},
     PAGES = {x+334},
}

@article {Ls91,
    AUTHOR = {Lusztig, George},
     TITLE = {Quivers, perverse sheaves, and quantized enveloping algebras},
   JOURNAL = {J. Amer. Math. Soc.},
  FJOURNAL = {Journal of the American Mathematical Society},
    VOLUME = {4},
      YEAR = {1991},
    NUMBER = {2},
     PAGES = {365--421},
}

@incollection {Hop71,
    AUTHOR = {Hopcroft, John},
     TITLE = {An {$n$} log {$n$} algorithm for minimizing states in a finite
              automaton},
 BOOKTITLE = {Theory of machines and computations ({P}roc. {I}nternat.
              {S}ympos., {T}echnion, {H}aifa, 1971)},
     PAGES = {189--196},
 PUBLISHER = {Academic Press, New York-London},
      YEAR = {1971},
}

@incollection {Rut00,
    AUTHOR = {Rutten, Jan J. M. M.},
     TITLE = {Universal coalgebra: a theory of systems},
      NOTE = {Modern algebra and its applications (Nashville, TN, 1996)},
   JOURNAL = {Theoret. Comput. Sci.},
  FJOURNAL = {Theoretical Computer Science},
  SERIES = {Selected papers from the Computer Science Sessions of the International Conference held at Vanderbilt University, Nashville, TN, May 1996},
BOOKTITLE = {Theoretical Computer Science},
PUBLISHER = {American Elsevier Publishing Company},
    VOLUME = {249},
      YEAR = {2000},
    NUMBER = {1},
     PAGES = {3--80},
}

@article {WDMS20,
    AUTHOR = {Wissmann, Thorsten and Dorsch, Ulrich and Milius, Stefan and
              Schr\"oder, Lutz},
     TITLE = {Efficient and modular coalgebraic partition refinement},
   JOURNAL = {Log. Methods Comput. Sci.},
  FJOURNAL = {Logical Methods in Computer Science},
    VOLUME = {16},
      YEAR = {2020},
    NUMBER = {1},
     PAGES = {Paper No. 8, 63},
}

@incollection {ABHKMS12,
    AUTHOR = {Ad\'amek, Ji\v r\'i{} and Bonchi, Filippo and H\"ulsbusch,
              Mathias and K\"onig, Barbara and Milius, Stefan and Silva,
              Alexandra},
     TITLE = {A coalgebraic perspective on minimization and determinization},
 BOOKTITLE = {Foundations of software science and computational structures},
    SERIES = {Lecture Notes in Comput. Sci.},
    VOLUME = {7213},
     PAGES = {58--73},
 PUBLISHER = {Springer, Heidelberg},
      YEAR = {2012},
}


\end{document}